\documentclass[12pt, bibother, otherbib]{article}

\usepackage{isabelle,isabellesym}
\usepackage[utf8]{inputenc} 
\usepackage[T1]{fontenc}    
\fontencoding{T1}\selectfont
\usepackage{hyperref}       
\usepackage{url}            
\usepackage{booktabs}       
\usepackage{amsfonts}       
\usepackage{nicefrac}       
\usepackage{lipsum}

\usepackage{amsmath,amssymb}
\usepackage{xpatch}

\newcommand{\arrows}{\longrightarrow}
\newcommand\narrows{
  \mathrel{\mkern2.7mu\not\mkern-2.7mu\longrightarrow}}

\newcommand{\eop}{\bigstar}  

\newcommand{\FF}{{\cal F}}

\newcommand{\PP}{{\cal P}}

\newcommand{\triple}[3]{\langle\kern1pt#1 , \:#2 , 
      \:#3 \kern1pt\rangle }
\newcommand{\trpl}[3]{\langle\kern1pt#1 , \:#2 , 
      \:#3 \kern1pt\rangle }

\newcount\skewfactor
\def\mathunderaccent#1#2 {\let\theaccent#1\skewfactor#2
\mathpalette\putaccentunder}
\def\putaccentunder#1#2{\oalign{$#1#2$\crcr\hidewidth
\vbox to.2ex{\hbox{$#1\skew\skewfactor\theaccent{}$}\vss}\hidewidth}}

\newcommand{\stevo}{Todor{\v{c}}evi{\'c}}
\newcommand{\erd}{Erd\H{o}s}
\let\ts=\thinspace
\newcommand{\bars}[1]{\lvert#1\rvert}
\newtheorem{Theorem}{Theorem}[section]
\newtheorem{theorem}{Theorem}[section]

\newtheorem{lemma}[Theorem]{Lemma}

\newtheorem{observation}[Theorem]{Observation}
\newenvironment{proof}
{\noindent{\bf Proof.}}{\par\bigskip}

\newtheorem{Definition}[Theorem]{Definition}

\newtheorem{notation}[Theorem]{Notation}

\title{Formalising Ordinal Partition Relations\\ Using Isabelle/HOL}

\begin{document}
\author{Mirna D\v zamonja (\texttt{mdzamonja@irif.fr})\\
     IRIF\\
     CNRS-Université de Paris\\
     6 Rue Nicole-Reine Lepaute\\
75013 Paris, France\\
 \and
Angeliki Koutsoukou-Argyraki 
 (\texttt{ak2110@cam.ac.uk})
 \and
 and Lawrence C. Paulson FRS (\texttt{lp15@cam.ac.uk})\\
Computer Laboratory, University of Cambridge\\
   15 JJ Thomson Avenue, Cambridge CB3 0FD, UK}

\maketitle

\begin{abstract}This is an overview of a formalisation project in the proof assistant Isabelle/HOL of a number of research results in infinitary combinatorics and set theory (more specifically in ordinal partition relations) by Erd\H{o}s--Milner, Specker, Larson and Nash-Williams, leading to Larson's proof of the unpublished result by E.C. Milner 
asserting that for all $m \in \mathbb{N}$, $\omega^\omega\arrows(\omega^\omega, m)$.
This material has been recently formalised by Paulson and is available on the Archive of Formal Proofs; here we discuss some of the most challenging aspects of the formalisation process. 
This project is also a demonstration of working with Zermelo--Fraenkel set theory in higher-order logic.

\end{abstract}


\emph{AMS 2020 subject classes}: 03E02, 03E05, 03E10, 03B35, 68V20, 68V35.

\section{Introduction}\label{sec:intro}

Higher-order logic theorem proving was originally intended for proving the correctness of digital circuit designs. The focus at first was bits and integers. But in 1994, a bug in the Pentium floating point division unit cost Intel nearly half a billion dollars~\cite{nicely-pentium-fdiv}, and verification tools suddenly needed a theory of the real numbers. The formal analysis of more advanced numerical algorithms, e.g.\ for the exponential function~\cite{harrison-exp}, required derivatives, limits, series and other mathematical topics. Later, the desire to verify probabilistic algorithms required a treatment of probability, and therefore of Lebesgue measure and all of its prerequisites~\cite{hoelzl-three}. The more verification engineers wanted to deal with real-world phenomena, the more mathematics they needed to formalise. And it was the spirit of the field to reduce everything to a minimal foundational core rather than working on an axiomatic basis.
  So a great body of mathematics came to be formalised in higher-order logic, including advanced results such as the central limit theorem~\cite{avigad-clt}, the prime number theorem~\cite{harrison-pnt}, the Jordan curve theorem~\cite{hales-jordan-curve} and the proof of the Kepler conjecture~\cite{hales-formal-Kepler}.   
  
 The material we have formalised for this case study \cite{erdos-theorem-partition,Jean,MR2603812} comes from the second half of the 20th century and concerns an entirely unexamined field: infinitary combinatorics and more specifically, ordinal partition relations. This field deals with generalisations of Ramsey's theorem to transfinite ordinals. It was of special interest to the legendary Paul Erd\H{o}s, and it is particularly lacking in intuition, to such an extent that even he could make many errors \cite{erdos-theorem-partition-corr}. Moreover, because our example requires ordinals such as $\omega^\omega$, it is a demonstration of working with Zermelo-Fraenkel set theory in higher-order logic.
  
 From a mathematical point of view, this work has the ambitious aim to be more than a case study of formalisation. We hope that it is a first step in a programme of finding ordinal partition relations by new methods, using the techniques developed in the formalisation. The reader familiar with Ramsey theory on cardinal numbers, including the encyclopaedic
book by Paul \erd{} et~al.\ \cite{erdoshajnalmaterado} or more recent works on the use of partition relations in topology and other celebrated applications, e.g.\ by Stevo \stevo{} \cite{Todorcevicpairs,stevoOCA, MR2603812}, might be doubtful about the need for new methods in discovering partition relations. But none of the powerful set-theoretic methods for studying \emph{cardinal} partition relations apply to \textit{ordinal} partition relations. The difficulty is that in addition to the requirement on the monochromatic set to have a given size, which we would ask of a cardinal partition relation, the ordinal case also has the requirement of preserving the order structure through having a fixed order type. In fact, the difference between the order structure versus an unstructured set shows up already in the case of addition: the addition of infinite cardinals is trivial, whereas with ordinals we don't even have $1+\alpha=\alpha+1$. Ordinal partition relations are the first instance of structural Ramsey theory, which is a growing and complex area of combinatorics. 

The fact is that everything we know about ordinal partition relations---which is short enough to be reviewed in our \S\ref{sec:ordinal_partitions}---has been proven painfully and laboriously. Ingenious constructions by several authors since the 1950s have chipped the edges off the most important problem in the subject, which is to characterise the countable ordinals $\alpha$ such that
$\alpha\arrows(\alpha, m)$ for a given natural number $m$ (for the notation see \S\ref{sec:ordinal_partitions}). The simplest nontrivial case of $m=3$ is the subject of a 
\$1000 open problem of Erd\H{o}s, posed back in 1987, see \S\ref{sec:ordinal_partitions}. 

We may ask why it is that modern set theory is so silent on the subject of ordinal partitions. Perhaps it is the case of the chicken and the egg. In the case of cardinal numbers, whose study has been at the heart of almost everything done in set theory since the time of Cantor, one of the first important advances was exactly the understanding of partition relations. They are the
backbone of infinite combinatorics: many theorems in set theory can be formulated in terms of partitions, an attitude well supported by the work of \stevo{} cited above. So to better understand the ordinals, we \emph{first} have to understand their partition properties, rather than expecting that powerful general methods will be developed first and then yield an understanding of ordinal partitions. The most interesting problems about ordinal partitions are about countable ordinals, while modern combinatorial set theory really only starts at the first uncountable cardinal. Even the popular method of using countable elementary submodels is only useful if one applies it to cardinals, since the intersection of a countable elementary submodel $M$ with the ordinals is a countable highly indecomposable limit $\delta$, such that $\delta$ is actually a subset of $M$ and $\delta=\omega_1^M$. So $M$ reflects everything about ordinals below $\delta$ and nothing about the ones above, not giving any room for reflection arguments that elementary submodels are used for. Consequently, it cannot be used to argue about countable ordinals, as much as it can be used to argue about 
$\omega_1$. 

The notation in the paper is quite standard, while every new notion is explained in the relevant part of the paper. Throughout we use the identification of the ordinal $\omega$ with the set of natural numbers and of each natural number $n>0$ with the set of
its predecessors $\{0,1,\ldots, n\}$. This is also equal to $\omega\setminus (n+1)$.

The plan of this paper is as follows: in the next section we give a short but comprehensive introduction to ordinal partition relations; in Section\ts \ref{sec:theorems}, we present the material formalised including  brief sketches of 
Larson's proofs of partition theorems for $\omega^2$ and $\omega^\omega$; in Section\ts\ref{sec:nw}, we give a more detailed exposition on the Nash-Williams partition theorem including two different proofs; in Section\ts\ref{sec:isabelle}, we present an introduction to the Isabelle/HOL theorem prover; Section\ts\ref{sec:formalisation-nw} presents our formalisation of the Nash-Williams theorem; Section\ts\ref{sec:formalisation-jean} discusses the formalisation of Larson's proofs; finally, Section\ts\ref{sec:conclusion} summarises what we learned from this project.

\section{Ordinal partition relations}\label{sec:ordinal_partitions}

As a side lemma in his work on decidability, Frank Ramsey \cite{Ramseyth} in 1929 proved what we now call Ramsey's theorem. It states that for any two natural numbers $m$ and $n$, if we divide the unordered $m$-tuples of an infinite set $A$ into $n$ pieces, there will be an infinite subset $B$ of $A$ whose all unordered $m$-tuples are in the same piece of the division. This theorem has since been generalised in many directions and Ramsey theory now forms an important part of combinatorics, both in the finite and in the infinite case. We shall only discuss Ramsey theory of cardinal and ordinal numbers, although a vast theory exists extending Ramsey theory to various structures, see for example the work of \stevo{} \cite{stevoOCA,MR2603812}. Also, we shall only be interested in partitions of pairs of ordinals, as it already proves to be quite challenging. In this section we review what is known about such ordinal partitions.

It is convenient to introduce the notation coming from \erd{} and his school, for example in \textit{Combinatorial Set Theory} \cite{erdoshajnalmaterado}. Writing
\begin{equation}\label{def:ordpart}
\alpha\arrows (\beta, \gamma)
\end{equation}
means that for every partition of the set $[\alpha]^2$ of unordered pairs of elements of $\alpha$ into two parts (called colours, say 0 and 1), there is either a subset $B$ of $\alpha$ of order type $\beta$ whose pairs are all coloured by 0,
or a subset $C$ of $\alpha$ of order type $\gamma$ whose pairs are all coloured by 1. Such a $B$ is said to be
\textit{0-monochromatic}, while $C$ is \textit{1-monochromatic}. The notation implies that the situation is trivial unless $\beta$, $\gamma\le\alpha$, as we shall assume. Note also that for all ordinals, the following rules of monotonicity hold:
\[
\text{if } \alpha\arrows (\beta, \gamma)\text{ then } \alpha'\arrows (\beta', \gamma')
\text { for } \alpha'\ge\alpha,\, \beta'\le\beta,\, \gamma'\le\gamma.
\]
Some authors, such as Larson~\cite{Jean}, use the notation $\alpha\arrows (\beta, \gamma)^2$ to emphasise that it is pairs that are coloured, but since we shall only ever work with pairs, we omit the superscript. The negation of $\alpha\arrows (\beta, \gamma)$ is written $\alpha\narrows (\beta, \gamma)$.

It turns out that finding the triples $\alpha,\beta,\gamma$ for which
the relation (\ref{def:ordpart}) holds is highly non-trivial. It is even non-trivial when all of the ordinals
$\alpha,\beta,\gamma$ are countable, which is the case to which we shall restrict our attention. On the other hand, for
$\alpha\le\omega$ the situation is already understood by the classical Ramsey theorem, so we shall assume
$\alpha>\omega$.
We should also assume that $\gamma$ is finite \cite[p.177]{Hajnal-Larson} (note that all arithmetic in the paper is ordinal arithmetic):

\begin{observation}\label{gammafinite} For any ordinal $\alpha>\omega$ we have $\alpha\narrows (|\alpha|+1, \omega)$.
\end{observation}

By definition we have $\alpha\arrows(\alpha,2)$ for any $\alpha$, so the first nontrivial case is the following question, to which \erd{} attached a prize of \$1000 in 1987~\cite{Erdoslist}:

\bigskip\noindent
\bf{\erd's problem} \rm Characterise the set of all countable ordinals $\alpha$ such that $\alpha\arrows(\alpha,3)$.
\bigskip

\noindent
This problem is still very much open. In fact, the more general problem of characterising countable ordinals $\alpha$ and natural numbers $m$ such that 
\[
\alpha\arrows (\alpha, m)
\]
holds was asked already by Erd\H{o}s and Richard Rado in 1956 \cite{ErdosRado} and it was the slow progress on it that made \erd{} reiterate the simplest case of this problem in his 1987 problem list \cite{Erdoslist}.

The first progress towards solving \erd's problem came from Ernst Specker \cite{Speckerord}, whose results were continued by  Chen-Chung Chang \cite{Changord}. Chang gave a very involved proof of
$\omega^\omega\arrows(\omega^\omega, 3)$ and gained \$250 from \erd. In an unpublished manuscript, Eric Milner improved Chang's result to say that
\[
\omega^\omega\arrows(\omega^\omega, m)
\]
for all natural numbers~$m$. The main proof that we have formalised is Jean Larson's proof of Milner's result. Comparing her paper~\cite{Jean} with earlier proofs explains why the paper is called `A short proof \ldots ', but it is not a short proof and formalising it was a challenge.

It turns out that the behaviour of countable ordinals with respect to partition relations is influenced by their Cantor Normal Form, so we take the opportunity to remind the reader of that concept.

\begin{theorem}[Cantor Normal Form]\label{Cantonormal} Every ordinal number can be written in a unique way as an ordinal sum of the form
\[
\omega^{\beta_0}\cdot m_0 + \omega^{\beta_1}\cdot m_1 +\ldots  \omega^{\beta_n}\cdot m_n,
\]
where $n$ is non-negative integer, as are the~$m_i$ for $i\le n$, and $\beta_0>\beta_1>\ldots\beta_n\ge 0$ are ordinals.
\end{theorem}

The state of the art regarding the known positive instances of \erd's problem is the following theorem of Rene Schipperrus from his 1999 Ph.D. thesis \cite{Schipperusthesis}, published many years later, in 2010, as a journal paper \cite{Schipperusjournal}:

\begin{theorem}[Schipperus 1999] Suppose that $\beta$ is a countable ordinal whose Cantor Normal Form has at most two summands. Then
\[
\omega^{\omega^\beta}\arrows(\omega^{\omega^\beta}, 3).
\]
\end{theorem}

The delay between the thesis and the paper is indicative of the difficulty of the proof and the process of checking its correctness. A sketch of Schipperus' proof, divided in seven subsections, is given on pages 188--209 of the excellent survey article \cite{Hajnal-Larson} by Andr\' as Hajnal and Larson. The reason that Schipperus focused on ordinals of the type
$\omega^{\omega^\beta}$ is that if the ordinal $\alpha$ is not a power of $\omega$ then it cannot satisfy
$\alpha\arrows(\alpha,3)$, as shown in the following Observation \ref{decomposable}. Hence, only the powers of $\omega$ are of interest.
Fred Galvin showed~\cite{Galvin} that for an ordinal of the form $\alpha=\omega^\beta$ where $\beta\ge 2$ is not itself a power of
$\omega$, we have
$\alpha\narrows(\alpha,3)$. Hence Schipperus' choice of ordinals. Still open is the case of $\alpha=\omega^{\omega^\beta}$
where $\beta$ has at least three summands in its Cantor Normal Form.

\begin{observation}\label{decomposable} Suppose that $\alpha$ is an ordinal which is not a power of $\omega$. Then
$\alpha\narrows(\alpha,3)$.
\end{observation}

\begin{proof} It follows from the Cantor Normal Form that any $\alpha$ which is not a power of $\omega$ is additively decomposable: there exist ordinals $\beta$, $\gamma<\alpha$ such that $\alpha=\beta+\gamma$.
Fixing such $\beta$ and $\gamma$, we define $c$ on $[\alpha]^2$ by letting $c(x,y)=0$ if either $x$, $y<\beta$ or
$x$, $y\ge \beta$. Otherwise, we let $c(x,y)=1$.

Then it suffices to note that any 0-monochromatic subset of $\alpha$ is either contained in $\beta$ or in $[\beta,\beta+\gamma)$ and hence has the order type at most $\max(\beta,\gamma)$, which is strictly less than $\alpha$. On the other hand, if we have distinct $x$, $y$, $z<\alpha$, at least two of them will be $<\beta$ or at least two of them will be $\ge\beta$ and in either case, the set they form will get mapped to~0 by $c$. Hence the set $\{x,y,z\}$ is not 1-monochromatic.
$\eop_{\ref{decomposable}}$
\end{proof}

In this section, we have mostly concentrated on the result we formalised and \erd's problem. Information on some additional instances of $\alpha\arrows(\alpha,m)$ for $m>3$ can be found in the Hajnal-Larson paper \cite{Hajnal-Larson}.

\section{Theorems formalised}\label{sec:theorems}

The ultimate objective of the project was to formalise Larson's proof \cite{Jean} of the following unpublished result by E.C. Milner:
\begin{theorem}\label{larsonresult}
For all $m \in \mathbb{N}$, $\omega^\omega\arrows(\omega^\omega, m)$.
\end{theorem}
While working towards that objective, many set-theoretic prerequisites had to be formalised, notably Cantor Normal Form, indecomposable ordinals and many elementary properties of order types. Her paper contains, as a simpler example of the methods she employed, a proof of Specker's theorem \cite{Speckerord}: 
\begin{theorem}[Specker]\label{specker} For all $m<\omega$, $\omega^2\arrows(\omega^2, m)$\end{theorem}
 Although not strictly necessary, the proof of Theorem \ref{specker}
  was formalised as a warmup exercise, as it is structured similarly to the proof of Theorem \ref{larsonresult}. 
 The project also required the formalisation of a short but difficult (and error filled) proof by Erd\H{o}s and Milner \cite{erdos-theorem-partition-corr}: 
\begin{theorem}[Erd\H{o}s-Milner]\label{erdosmilner} For all $n<\omega$ and for all $\alpha <\omega_1$,
$$\omega^{1+\alpha \cdot n}\arrows(\omega^{1+\alpha}, 2^n).$$
\end{theorem} 

The last significant side project necessitated by Larson's proof of Theorem~\ref{larsonresult} was to formalise the Nash-Williams partition theorem, as presented by \stevo~\cite{MR2603812}.
The main objects in Nash-Williams' theorem are families of finite subsets of $\omega$. We introduce some notation and definitions regarding such sets.

\begin{notation} (1) Let $A$ be any subset of $\omega$. We write $[A]^{<\omega}$ for the set of all finite subsets of $A$ and $[A]^{\infty}$ for the set of all infinite subsets of $A$. 
For an integer $k\ge0$, let $[A]^k$ be the set of all the $k$-element subsets of~$A$.

{\noindent (2)} We identify sets in $[\omega]^{<\omega}$ with their increasing enumerations, and hence
the set $[\omega]^{<\omega}$ becomes a subset of the set ${}^{<\omega}\omega$ of finite sequences in
$\omega$. Therefore, we can  consider the relation of being an initial segment on $\PP(\omega)$, writing $s\sqsubseteq t$ when $s$ is an initial segment of $t$.
\end{notation}

We shall be interested in subsets $\FF$ of $[\omega]^{<\omega}$ that are \emph{dense} in the sense that every element of
$[\omega]^{\infty}$ has an element of $\FF$ as an initial segment. In particular, we shall consider such sets that are minimal, meaning that we cannot take away an element of $\FF$ and still satisfy the density requirement. This is the same as to say that if we have an element of $\FF$ then none of its proper initial segments are in $\FF$. Larson \cite{Jean} calls the sets given by the latter requirement, \emph{thin} and Todor{\v c}evi{\'c} \cite[Def.\ 1.1.2 (2)]{MR2603812} calls them Nash-Williams. 
Here is the formal definition.

\begin{Definition}\label{front} (Thin families)
 A \emph{thin} or \emph{Nash-Williams} family on an infinite set $A\subseteq \omega$  is a subset $\FF$ of $[A]^{<\omega}$ such that for every $s$, $t\in \FF$, if $s\sqsubseteq t$ then $s=t$. A set $A$ is \emph{thin} if for all $s$, $t \in A$, $s$ is not a proper initial segment of $t$.
 %
%
\end{Definition}


\begin{theorem}[Nash-Williams]\label{nashwilliams}
 For any infinite set $M\subseteq \omega$, for any thin set 
 $A$, for any function
 $h:\{s \in A: s \subseteq M\} \rightarrow \{0, 1\}$, there exists an 
 $i \in \{ 0, 1\}$
 and an infinite set $N \subseteq M$ so that 
 $h(\{ s \in A:s \subseteq N\}) \subseteq \{ i\}$. 
 \end{theorem}
 A more detailed analysis of the formalisation of the Nash-Williams theorem will be given in the following section.
 This theorem  is a generalisation of Ramsey's theorem  \cite{Ramseyth}: \footnote{The word \emph{generalisation} deserves explanation. It is often said, including by Larson~\cite{Jean}, that Nash-Williams theorem is a Ramsey-type statement, however this does not appear immediately. Namely, Ramsey's theorem applies to colourings of \emph{pairs} of elements of $\omega$, while the Nash-Williams theorem applies to colouring singletons in a thin set, not the pairs of the elements of it. The crude analogue of the  theorem applying to pairs of elements of a thin set is easily seen to be false, while however there exists a finer version of Nash-Williams theorem using fronts and the shift-initial segment relation, obtaining a result that does apply to pairs and $n$-tuples of the elements of a thin set.} 
\begin{theorem}[Ramsey]\label{ramsey}
 For every nonzero $p < \omega$, every infinite set $M \subseteq \omega$ 
 and every function $h: [M]^p \rightarrow \{0, 1\}$, there exists an $i \in \{ 0, 1\}$ and an infinite set $N \subseteq M$ so that 
 $h([N]^p) =\{i\}$.
 \end{theorem}
For simplicity, both theorems are presented in their 2-colour versions. We obtain Ramsey's theorem from Nash-Williams simply by noting that the set~$[M]^p$ is thin.
While Ramsey's theorem (Theorem \ref{ramsey}) was  used in the proof of Theorem \ref{specker}, the  
Nash-Williams partition theorem (Theorem \ref{nashwilliams}) was used in a similar fashion in the proof of Theorem \ref{larsonresult}.

We give brief sketches of the proofs of Theorems \ref{larsonresult} and \ref{specker} below. 
For details, the reader may refer to Larson's paper \cite{Jean} or to the formalised versions of the proofs, where every step is made explicit,
in the entries by Paulson at the Archive of Formal Proofs \cite{Nash_Williams-AFP, Ordinal_Partitions-AFP}.

\subsection{Sketch of Larson's proof of Specker's Theorem}

It is sufficient to prove Theorem~\ref{specker} for functions $f:[U]^2 \rightarrow \{0, 1\}$ where $U=\{ (a,b): a<b<\omega \}$ is ordered lexicographically and has order type $\omega^2$. The following is Larson's Definition~2.2 \cite{Jean}.
\begin{Definition}[Interaction Scheme]A pair $A=\{(a, b), (c, d)\}$ of elements from $U$ with $a \leq c$ is of
\emph{form~0} if $a<b<c<d$, \emph{form~1} if $a<c<b<d$,  \emph{form~2} if $a<c<d<b$ and \emph{form~3} if $a=c$ and $b \neq d$.
If $A$  has one of these forms, then the \emph{interaction scheme} of $A$ is defined by $$i(A) =\{a, b, c, d\}.$$
\end{Definition}
Let $m<\omega$ and $f:[U]^2 \rightarrow \{0, 1\}$ be given
so that there is no $m$-element set $M$ for which
$f([M]^2) =\{1\}$. It is enough to find a set $X \subseteq U$ order isomorphic to $\omega^2$ for which $f([X]^2) =\{0\}$.\\ \\
It follows from Ramsey's theorem (Theorem \ref{ramsey}) that we can obtain an infinite set $N \subseteq \omega$
and $j_0$, $j_1$, $j_2$, $j_3 \in \{0, 1\}$ so that for any $k<4$ and any pair $\{x, y\} \subset U$ of form $k$ with 
$i( \{x, y\}) \subseteq N$, we have $f( \{x, y\}) = j_k$.  Then it is shown that for  $N$ and each $k=0, 1, 2, 3$ we can construct four $m$-element sets $M_0$, $M_1$, $M_2$, $M_3$ such that for any $k<4$ and any pair $\{x, y \} \subseteq M_k$, $\{x, y \}$ has form $k$  and $i(\{x, y\}) \subseteq N$, so $f(\{x, y\}) = j_k$. As we had assumed that there is no $m$-element set $M$ for which $f([M]^2) =\{1\}$, there follows $j_0 = j_1 = j_2 = j_3 =0$. It is then shown that for the infinite set $N$ we may obtain a set $X \subseteq U$ which is order isomorphic to $\omega^2$ so that for any pair $\{x, y\} \subseteq X$ there is a $k<4$ so that $\{x, y\}$ has form $k$ and $i(\{x, y\}) \subseteq N$. Therefore, 
\[ f(\{x, y\}) = j_0 = j_1 = j_2 =j_3 =0. \]
Thus  we have shown that $f([X]^2) = \{0\}$.

\subsection{Sketch of Larson's proof of Milner's Theorem}
\label{sec:larsonresult}

The reader will notice that the proof of Theorem~\ref{larsonresult} follows a pattern similar to the one above.
For each $n<\omega$, define 
\[ W(n) =\{ (a_0, a_1,\ldots , a_{n-1}): a_0 < a_1 <\cdots< a_{n-1} < \omega \} \]
ordered lexicographically. $W(n)$ is  thus order isomorphic to $\omega^n$.
 Let $W = W(0) \cup W(1) \cup \cdots$ be ordered first by length of sequence and then lexicographically, so that $W$ is order isomorphic to $\omega^\omega$. It now suffices to prove the theorem for functions
 $f: [W]^2 \rightarrow \{0, 1\}$. 
 
 Some notational conventions:
 \begin{itemize}
	\item $s$ and $t$ denote increasing finite sequences of elements of $\omega$, and we write
 	 $s<t$ to mean every element of $s$ is less than every element of $t$.
 	\item $s \ast t$ denotes the concatenation of two finite sequences.
 	\item $|s|$ denotes the length of~$s$.
 	\item $n_k$ is the $k^\textrm{th}$ term in the enumeration of $N$
in increasing order.
\end{itemize}

 We now present Larson's Definition~3.5 \cite{Jean}.
 
\begin{Definition}[Interaction Scheme]\label{def:inter}
A pair $\{x, y\} \subset W$ is of \emph{form $0$} if $|x| =|y|$. Let $k<\omega$  with $k>0$.  A pair $\{x, y\} \subset W$ with $|x| <|y|$ is of \emph{form 
 $2k-1$} (\emph{form~$2k$}) if there are non-empty sequences $a_1$, $a_2$, $\ldots$, $a_k (a_{k+1})$, $b_1$, $b_2$, $\ldots$, $b_k$ and $c$ and $d$ 
 such that
 \begin{enumerate}
 	\item $x=a_1 \ast a_2 \ast \cdots \ast a_k (\ast a_{k+1})$,
 	\item $y =b_1 \ast b_2 \ast \cdots\ast b_k$,
 	\item $c = (|a_1|, |a_1| +|a_2|, \ldots , |a_1| +|a_2|+\cdots+|a_k|(+|a_{k+1}|))$,
 	\item $d =(|b_1|, |b_1|+|b_2|,\ldots ,|b_1|+|b_2|+\cdots+|b_k|)$,
 	\item $c<a_1 <d < b_1< a_2<b_2<\cdots<a_k<b_k(<a_{k+1})$. 

 \end{enumerate}
 If $\{x, y\} $ is of form $2k-1$ (form $2k$) and $a_1$, $a_2$, $\ldots$, $a_k (a_{k+1})$, $b_1$, $b_2$, $\ldots$, $b_k$, $c$
 and $d$ are as above, then we call 
 $$i(\{x, y\}) = c\ast a_1 \ast d \ast b_1 \ast a_2 \ast b_2 \ast\cdots\ast a_k * b_k (\ast a_{k+1}) $$
 the \emph{interaction scheme} of $ \{x, y\})$.
 \end{Definition}
Let $m<\omega$. The theorem is trivially true for $m=0$ and $m=1$, so we assume that $m>1$.  In a similar fashion as before, we assume that $f:[W]^2 \rightarrow \{0, 1\}$ is given so that there is no $m$-element set $M$ for which $f([M]^2) =\{1\}$. So, to prove the theorem, we will find a set $X$ of type $\omega^\omega$ for which 
$f([X]^2) =\{0\}$. 

To this end, from the Erd\H{o}s-Milner theorem (Theorem \ref{erdosmilner}), 
we infer that for every $k$, $n<\omega$ with $n>0$, $\omega^{n \cdot k} \arrows (\omega^n, k)$; then considering $f$ restricted to $W(n \cdot m)$ we obtain a set $W'(n)$ order isomorphic to $\omega^n$ for which $f([W'(n)]^2) =\{0\}$. 
Let $W' := W'(0) \cup W'(1) \cup  W'(2) \ldots$. It is order isomorphic to $\omega^\omega$ and $f$ has value zero on pairs of sequences in $W'$ of the same length. 

Larson \cite[p.\ts134]{Jean} remarks `Without loss of generality, we may assume that $W'$ is our original set $W$. Thus to prove the theorem, we must find a set $X \subseteq W$ of type $\omega^\omega$ for which $f$ has value zero on pairs of sequences of different lengths.' This identification of $W$ and $W'$ is possible because there is an order isomorphism between them that preserves lengths. Strictly speaking, the bulk of the elaborate construction is done using $W$, then finally mapped back to $W'$; this can be seen in the formal version \cite{Ordinal_Partitions-AFP}.

By applying the Nash-Williams partition theorem (Theorem \ref{nashwilliams}) to $f$, we can obtain an infinite set $N$
and a sequence $\{j_k:k<\omega\}$ so that for any $k<\omega$ with $k>0$ and any pair $\{x, y\} $ of form $k$ with $(n_k) < 
i(\{ x, y\})\subseteq N$,  $f(\{x, y\}) =j_k$.
 Then it can be shown that for each $k<\omega$ with $k>0$ we obtain an $m$-element set $M_k$, so that for any 
$\{x, y\} \subset M_k $ we have $f(\{x, y\}) =j_k$. Thus, for any $k<\omega$ with $k>0$ it follows that $j_k =0$.
It is then shown that we may obtain a set $X\subseteq W$ order isomorphic to $\omega^\omega$, so that for each $\{x, y\} \subseteq X$
there is an $l<\omega$ for which $\{x, y\}$  has form $l$
and if $l>0$, then $(n_l) < i(\{x, y\}) \subseteq N$. Thus, for pairs $\{x, y\} \subseteq X$ which are not of form 0, we have  $f(\{x, y\})= j_l =0$ for some $l$.
In the  case $l=0$,  for any pair of form 0 we have by assumption $f(\{x, y\}) =0$. So we have shown that $f([X]^2) =\{0\}$.

\section{A proof of the Nash-Williams theorem}\label{sec:nw}

In this section we shall give a proof of a fundamental theorem due to Crispin Nash-Williams \cite{NashWilliamsoriginal} which we stated as Theorem \ref{nashwilliams} above. This result was discovered while studying the notion of well quasi orders (wqo) $P$, notably distinguishing those that have the property that for every countable ordinal $\alpha$, the set ${P}^\alpha$ of all sequences of length
$\alpha$ from $P$ is wqo when ordered by the embeddability relation. Such orders are called bqo or better quasi orders~\cite{NashWilliamsoriginal}. Neither wqo nor bqo are relevant here, but the theorem proved by Nash-Williams is of use in the study of ordinal partition relations, as well as in many other contexts, for example reverse mathematics \cite{MR1141936,MR1428011}. In particular, the theorem was used by Larson in the paper we formalised~\cite{Jean}.

Nash-Williams' theorem has seen many different proofs. While above we have steered away from presenting full proofs about ordinal partition relations as they are too long, we do present a proof of Nash-Williams' theorem, to give the reader the flavour of the way that proofs are constructed in this subject. Many other proofs of the theorem are known: Alberto Marcone gives one \cite{MR1428011} and refers to several others, including one by Stephen G. Simpson using descriptive set theory and the notion of bad arrays, which is perhaps the most popular proof these days (it originates in the methods of Fred Galvin and Karel Prikry~\cite{GalvinPrikry} and Richard Laver~\cite{laver71}) and a proof using well-founded trees, which can be found in the survey paper \cite{CarroyPequignot} by Raphaël Carroy and Yann Pequignot. Paulson formalised Nash-Williams' theorem as described in~\S\ref{sec:formalisation-nw}. The
formalisation~\cite{Nash_Williams-AFP} corresponds to Todor{\v c}evi{\'c}'s presentation~\cite{MR2603812} of the original Nash-Williams proof, which is the proof we give. He uses a notion of  combinatorial forcing due to Galvin and Prikry~\cite{GalvinPrikry}. The proof of Nash-Williams' theorem using combinatorial forcing appeared as early as 1985 in a note by Ian Hodkinson on a Ph.D course given by Wilfrid Hodges
\cite{Hodkinsonnotes} (which the authors have specifically asked not to use as a primary reference), but in fact Galvin and Prikry \cite{GalvinPrikry} prove a stronger theorem by the same method.

The theorem basically says that if we divide a thin set on a set $A$ of natural numbers into a finite number of pieces, then one of them will contain a thin set on some infinite subset $A'$ of $A$. This is the analogue of the version of the pigeonhole principle which says that if an infinite set is divided into a finite number of pieces, then one of the pieces is infinite. Notice that 
Theorem \ref{nashwilliams} yields this by applying the theorem finitely many times.

Many results in infinite combinatorics can be seen as instances of the fact that one can do set-theoretic forcing over a countable family of dense sets without changing the underlying universe, as proved by Paul Cohen~\cite{Coh66} and explained more carefully by others \cite{fasttrack,kunen}. This approach is known as \emph{combinatorial forcing} and is used in the proof we present. We use capital letters close to the beginning of the Latin alphabet $B, C \ldots$ to denote infinite subsets of $A$ and lowercase letters close to $s$ such as $s,u,t$ to denote finite sequences in $A$. Here comes a key definition of this particular instance of combinatorial forcing:

\begin{Definition}\label{accepts/rejects} Comparable, accepts, rejects, decides:
\begin{enumerate}
	\item $s$ and $t$ are \emph{comparable} if either $s\sqsubseteq t$ or $t\sqsubseteq s$.
	\item $B$ \emph{accepts} $s$ if there is $t\in \FF$ comparable to $s$ such that $t\setminus s\subseteq B$ (equivalently, $t\subseteq s\cup B$). Moreover $B$ \emph{strongly accepts} $s$ if every $C\in [B]^\infty$ accepts $s$.
	\item $B$ \emph{rejects} $s$ if $B$ does not accept $s$.
	\item $B$ \emph{decides} $s$ if $B$ either strongly accepts $s$ or rejects $s$.
	\item If $\FF'$ is a subset of $\FF$, we make definitions similar to (1)-(4) taking $\FF'$ as a parameter, and then we
add the qualification \emph{with respect to $\FF'$} to the notion of accepting, rejecting and so on.%
\footnote{For readers familiar with forcing: we may see these definitions as coming from a forcing notion consisting of the
pairs $(s,B)$ with $\max(s)<\min (B)$ where the extension is given by $(s,B)\le (t,C)$ if $s\sqsubseteq t$, $C\subseteq B$ and $t\setminus s\subseteq
B$. This resembles Prikry or Mathias forcing. In Hodkinson's notes such pairs are called \emph{Prikry pairs} and the idea of using them in combinatorics comes from the Galvin-Prikry partition theorem \cite{GalvinPrikry}.}
\end{enumerate}
\end{Definition}

Let us make some simple observations about the notions introduced.

\begin{observation}\label{Lemma1.17} Let $B$ and $s$ be given.
\begin{enumerate}
	\item If $B$ rejects (strongly accepts) $s$, then so does every $C\in [B]^\infty$. It follows that the analogue is true for the notion of deciding.
	\item There exists $C\subseteq B$ which decides $s$ and where $\max(s)<\min (C)$.
	\item $B$ accepts (rejects, strongly accepts, decides) $s$ iff $B\setminus\max(s)$ accepts (rejects, strongly accepts, decides) $s$.
\end{enumerate}
\end{observation}

\begin{proof} {\bf (of Observation \ref{Lemma1.17})}. (1) is evident from the definitions. For~(2), let $C=B\setminus [\max(s)+1]$. If $C$ strongly accepts $s$, then $C$ is as required. If that $C$ does not strongly accept $s$, then there is $D\in [C]^\infty$ which rejects $s$ and then $D$ is as required.

For~(3), if $B$ accepts $s$ then there is $t\in \FF$ comparable with $s$ such  that $t\setminus s\subseteq B$.
But then it follows that  $t\setminus s\subseteq (B\setminus \max(s))$ since if $t\sqsubseteq s$ this is vacuously true, and if
$s\sqsubseteq t$ then $t\setminus s$ is disjoint from $\max(s)$. The rest of the cases are proved similarly.
$\eop_{\ref{Lemma1.17}}$
\end{proof}

\begin{lemma}\label{Lemma1.18} There is $B\in [A]^\infty$ which decides all its finite subsets.
\end{lemma}

\begin{proof} {\bf (of Lemma \ref{Lemma1.18})}. By recursion on $n<\omega$ we shall choose pairs $(s_n, A_n)$ so that
\begin{itemize}
\item $A_n\in [A]^\infty$ and $s_n\in [A]^{<\omega}$,
\item $A_n$ decides every subset of $s_n$,
\item $\max(s_n)<\min(A_n)$ and
\item $A_{n+1}\in [A_n]^\infty$.
\end{itemize}
We let $s_0=\emptyset$ and we choose $A_0$ using Observation \ref{Lemma1.17} (2). Given $(s_n, A_n)$, let
$s_{n+1}=s_n\cup \min(A_n)$ and let $A_{n+1}\in [A_n]^\infty$ be a set which decides every subset of $s_{n+1}$. Such a set is obtained by a finite sequence of applications of Observation \ref{Lemma1.17} (2).
By cutting off the first several elements of $A_{n+1}$, which we can do by applying Observation \ref{Lemma1.17} (1), we can assume that $\max(s_{n+1})<\min(A_{n+1})$.

At the end of this recursion, let $B=\bigcup_{n<\omega} s_n$. Since we have made sure that $|s_n|=n$ for every $n$, we can conclude that $B$ is infinite. If $s$ is a finite subset of $B$, then there is first $n$ such
that $s\subseteq s_n$. We have that $B\setminus \max(s_n)\subseteq A_n$ and therefore $B\setminus \max(s_n)$ decides $s$. By Observation \ref{Lemma1.17}(3), we conclude that $B$ decides $s$.
$\eop_{\ref{Lemma1.18}}$
\end{proof}

Let $B$ be as provided by Lemma \ref{Lemma1.18}. The final lemma we need is the following:

\begin{lemma}\label{Lemma1.19} If $s\subseteq B$ is strongly accepted by $B$, then $B$ strongly accepts $s\cup \{n\}$ for all but finitely many $n\in B$. In particular, there is $m$ such that $B\setminus m$ strongly accepts $s\cup \{n\}$ for all
$n\in B\setminus m$.
\end{lemma}

\begin{proof} {\bf (of Lemma \ref{Lemma1.19})} Suppose, for a contradiction, that there is $s\subseteq B$ such that the set
\[
C=\{n\in B\setminus [\max(s)+1]:\, B \mbox{ rejects }s\cup \{n\}\}.
\]
is infinite, since $B$ decides every $s\cup \{n\}$. Hence $C$, in particular, accepts $s$ by the assumption on $B$, as exemplified by some $t\in\FF$. If $t\sqsubseteq s$ then clearly $t\sqsubseteq s\cup \{n\}$ for any $n$ and $t\setminus s\subseteq C$, so a contradiction. Hence $s$ is a proper initial segment of $t$.
Let $n=\min(t\setminus s)$. We claim that $C$ accepts $s\cup \{n\}$, which will give a contradiction with the choice of $C$.
Indeed, $t\setminus (s\cup \{n\})\subseteq C$ and $t\in\FF$, so we are done with the first claim of the lemma.

The second claim follows by taking $m$ large enough so that $B$ strongly accepts $s\cup \{n\}$ for all $n>m$ and
then using the hereditary nature of strong acceptance, as per Observation \ref{Lemma1.17}(3).
$\eop_{\ref{Lemma1.19}}$
\end{proof}

We now go back to the proof of Theorem \ref{nashwilliams}. Let $\FF_i=c^{-1}(i)$ for $i<2$. Clearly, both $\FF_i$ are thin sets, so
all the observations and lemmas we proved about $\FF$ apply also to each $\FF_i$. In particular, by applying
Lemma \ref{Lemma1.18} and Lemma \ref{Lemma1.19} to $\FF_0$ we can find $A'\in [A]^\infty$ which decides every of its finite subsets with respect to $\FF_0$ and moreover, for every $s\in [A']^{<\omega}$ which $A'$ strongly accepts, $A'$ also strongly accepts $s\cup \{n\}$ for all $n\in A'\setminus [\max(s)+1]$. If $A'$ rejects $\emptyset$, then clearly no finite subset of $A'$ is in $\FF_0$ and hence we have $c(s)=1$ for every element of $[A']^{<\omega}\cap\FF$.

Now suppose that $A'$ strongly accepts $\emptyset$ with respect to $\FF_0$. It follows by the choice of $A'$ (and an inductive argument) that $A'$ strongly accepts all its finite subsets with respect to $\FF_0$. If there were to exist an
element $s\in[A']^{<\omega}$ with $s\in \FF_1$, then the strong acceptance by $A'$ of $s$ would yield a $t\in \FF_0$ (so $t\neq s$)
comparable with $s$, which is impossible since $\FF$ is thin. Therefore we have $c(s)=0$ for every element of $[A']^{<\omega}\cap\FF$.
$\eop_{\ref{nashwilliams}}$

\section{Introduction to Isabelle}\label{sec:isabelle}
Isabelle is an interactive theorem prover originally developed in the 1980s with the aim of supporting multiple logical formalisms. These include first-order logic (intuitionistic as well as classical) and higher-order logic as well as Zermelo Fraenkel set theory. However, the 1990s saw higher-order logic take a dominant role in the field of interactive theorem proving, particularly in hardware verification~\cite{harrison-exp,kalvala-hol,PVS96:CAV}. While Isabelle/ZF and Isabelle/HOL share the entire Isabelle code base (basic inference procedures, a sophisticated user interface, etc.), Isabelle/HOL \cite{isa-tutorial} has much additional automation: so much so that it's the best choice even for set theory.

Unlike proof assistants based on constructive type theories, Isabelle/HOL implements \emph{simple type theory}. Types can take types as parameters but not integers for example. The type of finite sequences (known as \emph{lists}) takes the component type as a parameter, but there is no type of $n$-element lists. Logical predicates form the basis of a basic typed set theory, where any desired set can be expressed by comprehension over a formula. We can define the \textit{set} of $n$-element lists where the elements are drawn from some other set. Thus, the simple framework given by types can be refined through the use of sets. 

 There are always trade-offs between expressiveness of a formalism and ease of automation. Reliance on a simple classical formalism frees us from the many technical difficulties which seem to plague constructive type theories, such as intensional equality, difficulties with the concept of set, and performance issues in space and time. 
 
Isabelle employs the time-honoured LCF approach~\cite{gordon-tactics-milner}. This architecture ensures soundness through the use of a small proof kernel that implements the rules of inference and has the sole right to declare a statement to be a theorem.
Upon this foundation, Isabelle/HOL provides many forms of powerful automation~\cite{paulson-from-lcf}:
\begin{itemize}
  \item \emph{Simplification}, i.e., systematic directed rewriting using identities. 
  \item Sophisticated \emph{logical reasoning} even with quantifiers.
  \item \emph{Sledgehammer}: strong integration with external theorem provers.
  \item Automatic \emph{counterexample finding} for many problem domains. 
\end{itemize}
We found that some of this automation works effectively with Larson's elaborate constructions on sequences.

\subsection{Simple type theory in Isabelle/HOL}

Isabelle's higher-order logic is closely based on Church's simple type theory~\cite{church40}. It includes the following elements:
\begin{itemize}
  \item \emph{Types} and type operators, for example \isa{int} (the type of integers), or \isa{$\alpha$\isasymRightarrow$\beta$} (the type of functions from $\alpha$ to~$\beta$), or
  \isa{$\alpha$\,list} (the type of finite sequences whose elements have type~$\alpha$). Note the postfix syntax: \isa{(int list) set} is the type of sets of lists of integers.
  \item \emph{Terms} built of constants, variables, $\lambda$-abstractions, and function applications.
  \item \emph{Formulas}: terms of the truth value type, \isa{bool} (Church's~$o$), with the usual logical connectives and quantifiers.
  \item The \emph{axiom of choice} (AC) for all types via Hilbert's operator $\epsilon x. \phi$, denoting some~$a$ such that $\phi(a)$ if such exists.%
 \footnote{An undefined value is simply regarded as underspecified. It has the expected type, and we always have $(\epsilon x. \phi) = (\epsilon x. \phi)$ for example.}
 The Isabelle syntax is \isa{SOME x.\ P x}.
\end{itemize}
The typed set theory essentially identifies sets with predicates, with type \isa{$\alpha$\,set} essentially the same as \isa{$\alpha$\isasymRightarrow bool}.
On this foundation, recursive definitions of types, functions and predicates/sets are provided through programmed procedures that reduce such definitions to primitive constructions and automatically prove the essential properties.  This basis is expressive enough for the formalisation of the numerous advanced results mentioned in the introduction.

\subsection{ZFC in simple type theory}

Since the \isa{set} type operator can be iterated only finitely many times,
simple type theory turns out to be weaker than Zermelo set theory (let alone ZF).
For work requiring the full power of ZFC, it is convenient to assume some version of the ZF axioms within Isabelle/HOL\@. 
The approach adopted here~\cite{ZFC_in_HOL-AFP} seeks a smooth integration between ZFC and simple type theory. We introduce a type \isa{V} (the type of all ZF sets) and then \isa{V set} is the type of classes. Our ZF axioms characterise the \emph{small} classes, those that can be embedded into elements of~\isa{V}\@.
The axiom of choice is inherited from Isabelle/HOL\@.
On this basis it is straightforward to define the usual elements of Cantor's paradise, including ordinals, cardinals, alephs and order types. We borrow large formal developments from the existing Isabelle/HOL framework: recursion on~$\in$ is just an instance of well-founded recursion, and it's easy to deduce that the type \isa{real}  corresponds to some element of~\isa{V} without redoing the construction of the real numbers.  

We define order types on wellorderings only (yielding ordinals). These  wellorderings can be defined on any Isabelle/HOL type: we can consider orderings defined on type \isa{nat} rather than on the equivalent ordinal, $\omega$. We started with a small library of facts about order types, which grew and grew in accordance with the demands of the case study.

The point of adopting Isabelle/HOL over Isabelle/ZF is the possibility of making use of its aforementioned automation (Sledgehammer) and the counterexample-finding tools (Nitpick~\cite{blanchette-nitpick} and Quickcheck~\cite{bulwahn-quickcheck}). 
The main drawback of doing set theory in Isabelle/HOL is the impossibility of working without AC: the axiom is inherently part of the framework. That drawback has no bearing on the present project, however.

\section{Formalising the Nash-Williams theorem}\label{sec:formalisation-nw}

Although the Nash-Williams partition theorem is only a minor part of the project, it's a key result and its formalisation is brief enough to present in reasonable detail. In the next section, we'll turn to Larson's paper.

\subsection{Preliminaries}

As we saw in Sect.\ts\ref{sec:nw} above, the theorem is concerned with sets of natural numbers.
Finite sets of natural numbers can be identified with ascending finite sequences;
typical treatments of Nash-Williams use sets, while Larson uses sequences. For sets $S$ and $T$, we write $S<T$ to express that every element of $S$ is less than every element of~$T$ (it holds vacuously if either set is empty). This is essentially the same statement as the $s<t$ mentioned in the previous section.
Our formalisation \cite{Nash_Williams-AFP} follows \stevo~\cite{MR2603812} as in the proof in Sect.\ts\ref{sec:nw}.

$S$ is an \textit{initial segment} of~$T$ if $T$ can be written in the form $S\cup S'$ such that $S<S'$, written \isa{S\ \isasymlless \ S'} in Isabelle syntax. 
\begin{isabelle}
\isacommand{definition}\ init\_segment\ ::\ "nat\ set\ \isasymRightarrow \ nat\ set\ \isasymRightarrow \ bool"\isanewline
\ \ \isakeyword{where}\ "init\_segment\ S\ T\ \isasymequiv \ \isasymexists S'.\ T\ =\ S\ \isasymunion \ S'\ \isasymand \ S\ \isasymlless \ S'"
 \end{isabelle}
 
The \textit{Ramsey} property expresses the conclusion of the theorem for the general case of $r$ components. 
Its definition for a family $\cal F$ of sets $\cal F$ and an integer~$r$ is straightforward. A partition of $\cal F$ into $r$ disjoint sets is expressed as a map $f:{\cal F}\to \{0,\ldots,r-1\}$. Partition~$j$ is expressed as the inverse image $f^{-1}(j)$, written \isa{f\ -`\ \{j\}} in Isabelle syntax.
 
Now the Ramsey property for $\cal F$ and~$r$ holds if
for every partition map $f$ and every infinite set~$M$, there exists an infinite $N\subseteq M$ and $i<r$ such that for all $j<r$, if $j\not=i$ then partition~$j$ does not contain any subsets of~$N$. (\isa{Pow\ N} is the powerset of \isa{N}\@.)
 \begin{isabelle}
\isacommand{definition}\ Ramsey\ ::\ "[nat\ set\ set,\ nat]\ \isasymRightarrow \ bool"\isanewline
\ \ \isakeyword{where}\ "Ramsey\ \isasymF \ r\ \isasymequiv\isanewline
\ \ \ \ \ \isasymforall f\ \isasymin \ \isasymF \ \isasymrightarrow \ \{..<r\}.\ \isasymforall M.\ infinite\ M\ \isasymlongrightarrow \isanewline 
\ \ \ \ \ \ \ \ \ (\isasymexists N\ i.\ N\ \isasymsubseteq \ M\ \isasymand \ infinite\ N\ \isasymand \ i<r\ \isasymand \ \isanewline
\ \ \ \ \ \ \ \ \ (\isasymforall j<r.\ j\isasymnoteq i\ \isasymlongrightarrow \ f\ -`\ \{j\}\ \isasyminter \ \isasymF \ \isasyminter \ Pow\ N\ =\ \{\}))"
\end{isabelle}
 
 Recall that a family $\cal F$ of sets is thin provided every element of  $\cal F$ is finite (expressed as \isa{\isasymF \ \isasymsubseteq \ \{X.\ finite X\}}) and it does not contain distinct elements $S$ and $T$ where one is an initial segment of the other (see Definition \ref{front}).

\begin{isabelle}
\isacommand{definition}\ thin\_set\ ::\ "nat\ set\ set\ \isasymRightarrow \ bool"\isanewline
\ \ \isakeyword{where}\ "thin\_set\ \isasymF \ \isasymequiv\isanewline
\ \ \ \ \isasymF \ \isasymsubseteq \ \{X.\ finite X\}\ \isasymand \ (\isasymforall S\isasymin \isasymF .\ \isasymforall T\isasymin \isasymF .\ init\_segment\ S\ T\ \isasymlongrightarrow \ S=T)"
\end{isabelle}

These definitions provide the necessary vocabulary to state the theorem, although its proof appears at the end of the development after those of all prerequisite lemmas.  
\begin{isabelle}
\isacommand{theorem}\ Nash\_Williams:\isanewline
\ \ \isakeyword{assumes}\ \isasymF :\ "thin\_set\ \isasymF "\ "r\ >\ 0"\ \isakeyword{shows}\ "Ramsey\ \isasymF \ r"
\end{isabelle}

We now formalise the concepts of \textit{rejecting}, \textit{strongly accepting} and \textit{deciding} a set, as described in Sect.\ts\ref{sec:nw}. We regard  $\cal F$ as fixed and say $M$ decides $S$, etc. Note that the identification of `$M$ accepts $S$' with `$M$ does not reject $S$' is formalised as an abbreviation rather than a definition, a distinction that affects proof procedures but makes no difference mathematically.

\begin{isabelle}
\isacommand{definition}\ comparables\ ::\ "nat\ set\ \isasymRightarrow \ nat\ set\ \isasymRightarrow \ nat\ set\ set"\isanewline
\ \ \isakeyword{where}\ "comparables\ S\ M\ \isasymequiv\isanewline
\ \ \ \{T.\ finite\ T\ \isasymand \ (init\_segment\ T\ S\ \isasymor \ init\_segment\ S\ T\ \isasymand \ T-S\ \isasymsubseteq \ M)\}"
\end{isabelle}

\begin{isabelle}
\isacommand{definition}\ "rejects\ \isasymF \ S\ M\ \isasymequiv \ comparables\ S\ M\ \isasyminter \ \isasymF \ =\ \{\}"\end{isabelle}
 
\begin{isabelle}
\isacommand{abbreviation}\ "accepts\ \isasymF \ S\ M\ \isasymequiv \ \isasymnot \ rejects\ \isasymF \ S\ M"
\end{isabelle}

$M$ strongly accepts $S$ provided all infinite subsets of $M$ accept~$S$.
\begin{isabelle}
\isacommand{definition}\isanewline
\ \ \ "strongly\_accepts\ \isasymF \ S\ M\ \isasymequiv \ \isasymforall N\isasymsubseteq M.\ rejects\ \isasymF \ S\ N\ \isasymlongrightarrow \ finite\ N"
\end{isabelle}
 
\begin{isabelle}
\isacommand{definition}\ "decides\ \isasymF \ S\ M\ \isasymequiv \ rejects\ \isasymF \ S\ M\ \isasymor \ strongly\_accepts\ \isasymF \ S\ M" \end{isabelle}
 
 \begin{isabelle}
\isacommand{definition}\ "decides\_subsets\ \isasymF \ M\ \isasymequiv\isanewline
\ \ \ \ \ \ \ \ \ \ \ \isasymforall T.\ T\ \isasymsubseteq \ M\ \isasymlongrightarrow \ finite\ T\ \isasymlongrightarrow \ decides\ \isasymF \ T\ M"
 \end{isabelle}
 
 \subsection{The proofs}
 
A great many obvious facts about these primitives can be proved automatically. But there are some nontrivial properties not mentioned in the text that require elaborate proofs. A key technique in this field is called \textit{diagonalisation}, involving the construction of a sequence of infinite sets
$M_0\supseteq M_1\supseteq\cdots M_k\supseteq \cdots$ from which something can be obtained.

The following proposition states that an infinite set $M$ can be refined to an infinite $N\subseteq M$ that decides all subsets of the given finite set~$S$. The formal proof is 49 lines long and involves an inductive construction along with multiple inductive subproofs.
 \begin{isabelle}
 \isacommand{proposition}\ ex\_infinite\_decides\_finite:\isanewline
\ \ \isakeyword{assumes}\ "infinite\ M"\ "finite\ S"\isanewline
\ \ \isakeyword{obtains}\ N\ \isakeyword{where}\ "N\isasymsubseteq M"\ "infinite\ N"\ "\isasymAnd T.\ T\isasymsubseteq S\ \isasymLongrightarrow \ decides\ \isasymF \ T\ N"
 \end{isabelle}

\stevo's Lemma 1.18 states that an infinite set $M$ can be refined to an infinite $N\subseteq M$ that decides all of its finite subsets.
He notes that it follows by some ``immediate properties'' of the definitions ``and a simple diagonalisation procedure''. The formal equivalent of his one-line remark is 190 lines.
Of this, nearly 90 lines are devoted to the diagonalisation argument, including the construction of 
$M_0\supseteq M_1\supseteq\cdots M_k\supseteq \cdots$ and then the set
$\{m_0,m_1,\ldots,m_k,\ldots\}$ of the corresponding least elements;
proving that this set is the desired $N$ takes up the remaining lines. 
Perhaps a shorter formal proof could be found, given a few more hints.

\begin{isabelle}
\isacommand{proposition}\ ex\_infinite\_decides\_subsets:\isanewline
\ \ \isakeyword{assumes}\ "thin\_set\ \isasymF "\ "infinite\ M"\isanewline
\ \ \isakeyword{obtains}\ N\ \isakeyword{where}\ "N\ \isasymsubseteq \ M"\ "infinite\ N"\ "decides\_subsets\ \isasymF \ N"
\end{isabelle}

\stevo's Lemma 1.19 states that $M$ strongly accepts $\{n\}\cup S$ for all but finitely many $n$ in $M$, under the given conditions.
\begin{isabelle}
\isacommand{proposition}\ strongly\_accepts\_1\_19:\isanewline
\ \ \isakeyword{assumes}\ acc:\ "strongly\_accepts\ \isasymF \ S\ M"\isanewline
\ \ \ \ \isakeyword{and}\ "thin\_set\ \isasymF "\ "infinite\ M"\ "S\ \isasymsubseteq \ M"\ "finite\ S"\isanewline
\ \ \ \ \isakeyword{and}\ dsM:\ "decides\_subsets\ \isasymF \ M"\isanewline
\ \ \isakeyword{shows}\ "finite\ \{n\ \isasymin \ M.\ \isasymnot \ strongly\_accepts\ \isasymF \ (insert\ n\ S)\ M\}"
\end{isabelle}
He gives a four-line proof that leaves out any details. The formal equivalent is 69 lines and bears little resemblance to the original,
despite starting with the key definition\footnote{\isa{\{Sup\ S<..\}} is the set of integers greater than $\max S$}
\begin{isabelle}
\isacommand{define}\ N\ \isakeyword{where}\ "N\ =\ \{n\ \isasymin \ M.\ rejects\ \isasymF \ (insert\ n\ S)\ M\}\ \isasyminter \ \{Sup\ S<..\}"	
\end{isabelle}
The formal proof is by contradiction, and the key claim \isa{rejects\ \isasymF \ S\ N} is never established but rather its negation, from which \isa{False} is tediously squeezed.
Perhaps an expert could find a much neater proof.

Lemma 1.19 turns out to be too weak for its intended use in the sequel. The following strengthening is necessary. It yields an infinite $N\subseteq M$ such that $N$ strongly accepts $\{n\}\cup S$ for all $n$ in $N$ such that $n > \max(S)$. \stevo{} emailed a six-line proof sketch (he also helped with 1.19); the formal proof is 167 lines.
It includes a diagonalisation argument preceded by 50 lines of elaborate preamble, using Lemma~1.19.
\begin{isabelle}
\isacommand{proposition}\ strongly\_accepts\_1\_19\_plus:\isanewline
\ \ \isakeyword{assumes}\ "thin\_set\ \isasymF "\ "infinite\ M"\isanewline
\ \ \ \ \isakeyword{and}\ dsM:\ "decides\_subsets\ \isasymF \ M"\isanewline
\ \ \isakeyword{obtains}\ N\ \isakeyword{where}\ "N\ \isasymsubseteq \ M"\ "infinite\ N"\isanewline
\ \ \ "\isasymAnd S\ n.\ \isasymlbrakk S\isasymsubseteq N;\ finite\ S;\ strongly\_accepts\ \isasymF \ S\ N;\ n\isasymin N; \ S\ \isasymlless \ \{n\}\isasymrbrakk \isanewline
\ \ \ \ \ \ \ \ \ \ \ \isasymLongrightarrow \ strongly\_accepts\ \isasymF \ (insert\ n\ S)\ N"
\end{isabelle}

The proof of Nash-Williams itself is given for the case of $r = 2$ and the informal text is 11 lines long.
Thanks to this more detailed proof, the formal equivalent is only 65 lines long, shorter than those of several supposedly obvious lemmas. The straightforward generalisation to $r>0$ by induction is, at 66 lines, slightly longer but follows a standard argument.
\begin{isabelle}
\isacommand{theorem}\ Nash\_Williams\_2:\isanewline
\ \ \isakeyword{assumes}\ "thin\_set\ \isasymF "\ \isakeyword{shows}\ "Ramsey\ \isasymF \ 2"
\end{isabelle}

\subsection{A short proof, in detail}

Major formal proofs are too long to include in full, so we present a trivial proof in order to illustrate the style. Isabelle proofs are written in a structured language  containing nested scopes in which variables may be introduced along with assumptions and local definitions. Each such scope shows some explicitly-stated conclusion, perhaps establishing intermediate results along the way. 

Here is a trivial \isacommand{lemma} stating that $\{n\}\cup S$ is an initial segment of~$T$
if and only if $S$ is an initial segment of~$T$ and $n\in T$, provided $S<\{n\}$ and $n\le x$ for all $x\in T \setminus S$. The keyword \isacommand{proof} sets up both directions of the equivalence. The right-to-left direction (appearing after the keyword \isacommand{next}) is more interesting. From the right-hand side we obtain some $R$ such that $T = S\cup R$ and $S<R$, and after a few calculations, we derive the left-hand side.

\begin{isabelle}
\isacommand{lemma}\ init\_segment\_insert\_iff:\isanewline
\ \ \isakeyword{assumes}\ Sn:\ "S\ \isasymlless \ \{n\}"\ \isakeyword{and}\ TS:\ "\isasymAnd x.\ x\ \isasymin \ T-S\ \isasymLongrightarrow \ n\isasymle x"\isanewline
\ \ \isakeyword{shows}\ "init\_segment\ (insert\ n\ S)\ T\ \isasymlongleftrightarrow \ init\_segment\ S\ T\ \isasymand \ n\ \isasymin \ T"\isanewline
\isacommand{proof}\ \isanewline
\ \ \isacommand{assume}\ "init\_segment\ (insert\ n\ S)\ T"\isanewline
\ \ \isacommand{then}\ \isacommand{have}\ "init\_segment\ (\{n\}\ \isasymunion \ S)\ T"\ \isacommand{by}\ auto\isanewline
\ \ \isacommand{then}\ \isacommand{show}\ "init\_segment\ S\ T\ \isasymand \ n\ \isasymin \ T"\isanewline
\ \ \ \ \isacommand{by}\ (metis\ Sn\ Un\_iff\ init\_segment\_def\isanewline
\ \ \ \ \ \ \ \ \ \ \ \ \ \ init\_segment\_trans\ insertI1\ sup\_commute)\isanewline
\isacommand{next}\isanewline
\ \ \isacommand{assume}\ rhs:\ "init\_segment\ S\ T\ \isasymand \ n\ \isasymin \ T"\isanewline
\ \ \isacommand{then}\ \isacommand{obtain}\ R\ \isakeyword{where}\ R:\ "T\ =\ S\ \isasymunion \ R"\ "S\ \isasymlless \ R"\isanewline
\ \ \ \ \isacommand{by}\ (auto\ simp:\ init\_segment\_def\ less\_sets\_def)\isanewline
\ \ \isacommand{then}\ \isacommand{have}\ "S\isasymunion R\ =\ insert\ n\ (S\ \isasymunion \ (R-\{n\}))\ \isasymand \ insert\ n\ S\ \isasymlless \ R-\{n\}"\isanewline
\ \ \ \ \isacommand{unfolding}\ less\_sets\_def\ \isacommand{using}\ rhs\ TS\ nat\_less\_le\ \isacommand{by}\ auto\isanewline
\ \ \isacommand{then}\ \isacommand{show}\ "init\_segment\ (insert\ n\ S)\ T"\isanewline
\ \ \ \ \isacommand{using}\ R\ init\_segment\_Un\ \isacommand{by}\ force\isanewline
\isacommand{qed}	
\end{isabelle}

 Intermediate results are inserted using the keyword \isacommand{have} and existential claims with \isacommand{obtain}, while the conclusion is presented using \isacommand{show}. Justifications are introduced with~\isacommand{by}. A justification can consist of a proof method such as \isa{auto}, or a series of proof methods, or a full proof structure enclosed within the brackets \isacommand{proof} and \isacommand{qed}. Thus the various proof elements (there are many others) can be nested to any depth.

Because proofs in Isabelle's Isar language are \textit{structured}, they can be much more readable than proofs in other theorem provers.
Explicit statements of the assumptions and conclusions make structured proofs more verbose than the tactic-style proofs that predominate with other proof assistants, but infinitely more legible. The ideal is not merely to formalise mathematical results but to create a  document that makes the proof clear to the reader, where---without having to trust the software---a knowledgeable reader could decide for herself whether the claims follow from the  assumptions.  A human mathematician would not like to see an incomprehensible
 `black box' proof.  The possibility to recreate a `traditional'  proof from Isabelle code does help the user feel more at ease.  

\subsection{On the length of formal proofs}

The \textit{de Bruijn factor} \cite{wiedijk-de-bruijn} is defined as the ratio of the size of the formal mathematics to the size of the corresponding mathematical exposition. It can be regarded as measuring the cost of formalisation.

Unfortunately, it's highly inexact. Mathematical writing varies greatly in its level of detail. 
The first ever de Bruijn factor was calculated for Jutting's translation of Landau's \textit{Grundlagen der Analysis}---on the construction of the complex numbers---into AUTOMATH\@.
\begin{quotation}
  An aspect which has not been mentioned so far is the ratio between the length of pieces of AUT-QE text and the length of the corresponding German texts. Our claim at the outset was that this ratio can be kept constant. \ldots As a measure of the lengths the number of stored AUT-QE expressions \ldots{} and (rough estimates of) the number of German words.~\cite[p.\ts46]{jutting77}
\end{quotation}
The highest ratio, 6.4, is obtained for the translation of Landau's Chapter 4, in which the real numbers are constructed from the positive real numbers (in the form of Dedekind cuts, which had been defined in Chapter 3). But this is a book that devotes 173 pages to the development of complex number arithmetic from logic. The proof that $a/b=c/d \iff ad=bc$ includes two references to previous theorems about complex arithmetic. The material could be covered in 10\% of the space.

Formal proofs can also be more or less compact, the price of compactness generally being a loss of legibility. Wiedijk \cite{wiedijk-de-bruijn} has proposed to deal with some of the arbitrariness by comparing the sizes of compressed text files, but this requires retyping possibly lengthy texts into \LaTeX. He and others report de Bruijn factors in the range of 3--6, but for the Nash-Williams proof above it is 20 and upwards (crudely counting lines rather than symbols). Clearly one reason is that the source text is highly concise, only sketching out the key points. One of the lemmas (the strengthening of \stevo's~1.19) is not even stated. But also, our proof style is more prolix than is strictly necessary.

\section{Formalising Larson's proof}\label{sec:formalisation-jean}

Larson's proof \cite[pp.\ts133--140]{Jean} of Theorem \ref{larsonresult} is an intricate tour de force and the formal proof development is almost 4600 lines long. Here we can only cover its main features, focusing on a few typical constructions and some particular technical issues. We also  take a brief look at a few of the formal theorem statements.
Our objective is simply to highlight aspects of the formalisation task and the strengths and weaknesses of today's formal verification tools.

\subsection{Preliminary remarks} \label{sec:prelim-jean}

Much of the task of formalisation is simply labour: translating the definitions and arguments of the mathematical exposition into a formal language and generating proofs using the available automated methods. Particular difficulties arise when the exposition appeals to intuition or presents a construction. In the case of a construction there are two further sources of difficulties: when properties of the construction are claimed without further argument (as if the construction itself were sufficient proof), or worse, when proofs  later in the exposition depend on properties of the construction that are never even stated explicitly.

An example of appeal to intuition is when Larson constructs the set $W'$ and remarks that `without loss of generality'---a chilling phrase to the formaliser---we can regard it as the same as $W$. This turns out to mean that the bulk of the argument will be carried on using~$W$, which is simply the set of increasing sequences of integers. It is presumed obvious to the reader that none of the main lemmas could possibly be proved using~$W'$, about which little can be known. $W'$ is introduced towards the end, and the necessary adaptations to the main proof aren't that hard to figure out. But a little hint would have been helpful.

For an example of properties claimed without argument, consider the notation $i(\{x, y\})$ from Definition~\ref{def:inter}, suggesting that $i(\{x, y\})$ is uniquely determined by $x$ and~$y$. And so it turns out to be, though the formal version makes explicit its dependence on the `form' of $\{x, y\}$, namely~$l$. That $i_l(\{x, y\})$ is well-defined is perhaps obvious, since the decomposition of $x$ and $y$ into concatenations of sequences turns out to be unique, but the formal proof of this `obvious' fact is an elaborate induction, around 200 lines long. The function is also injective, which is never claimed explicitly but required for the proof of Lemma 3.6, which defines a function $g_k$ by $g_k (i(\{x, y\})) = g_k (\{x, y\})$. Proving this additional fact requires another substantial formal proof (100 lines).

For another example of the difficulty of formalisation, consider the following construction, typical of this problem domain. We are given a positive integer~$k$ and an infinite set $N = \{n_i : i < \omega\}$ of natural numbers, where the $n_i$ are an increasing enumeration of~$N$. Larson \cite[Lemma 3.7]{Jean} defines sequences $d^1$, $d^2$, \ldots, $d^m$ and $a^1_1$, $a^1_2$, \ldots, $a^1_{k+1}$, $a^2_1$, \ldots, $a^m_1$, \ldots $a^m_{k+1}$ as follows:
\begin{quote}
	Let $d^1 = (n_1,n_2, \ldots, n_{k+1}) = (d^1_1,d^1_2, \ldots, d^1_{k+1})$ and let $a^1_1$ be the sequence of the first $d^1_1$ elements of~$N$ greater than~$d^1_{k+1}$. Now suppose we have constructed $d^1$,  $a^1_1$, \ldots, $d^i$, $a^i_1$. Let $d^{i+1} = (d^{i+1}_1, \ldots, d^{i+1}_{k+1})$ be the first $k+1$ elements of~$N$ greater than the last element of~$a^i_1$, and let $a^{i+1}_1$ be the first $d^{i+1}_1$ elements of~$N$ greater than~$d^{i+1}_{k+1}$. This defines $d^1$, $d^2$, \ldots, $d^m$, $a^1_1$, $a^2_1$ \ldots, $a^m_1$. Let the rest of the sequences be defined in the order that follows, so that for any $i$ and~$j$, $a^i_j$ is the sequence of the least $(d^i_j - d^i_{j-1})$ elements of~$N$ all of which are larger than the largest element of the sequence previously defined:
	\[ (a^m_1)a^1_2, a^2_2, a^3_2, \ldots, a^m_2, a^1_3, \ldots, a^m_3,\ldots a^1_k,\ldots a^m_k, a^m_{k+1}, a^{m-1}_{k+1}, \ldots a^1_{k+1}. \]
\end{quote}
This construction is carefully crafted, particularly in the reversal at the end: $a^{m-1}_k, a^m_k, a^m_{k+1}, a^{m-1}_{k+1}$. The point is to achieve the conclusion of Larson's Lemma~3.7, namely that if $m$ and~$l$ are natural numbers with $l>0$, then there is an $m$ element set~$M$ such that for every $\{x,y\}\subseteq M$, $\{x,y\}$ has form~$l$ and $i(\{x,y\})\subseteq N$.

It turns out that if $l=2k-1$ then we can put
\[ M = \{a^i_1 * a^i_2 * \cdots * a^i_k : 1\le i\le m\}, \]
while if $l=2k$ then we can put
\[ M = \{a^i_1 * a^i_2 * \cdots * a^i_{k+1} : 1\le i\le m\}. \]
The reversal noted above turns out to be crucial to the second case, where $x$ is the concatenation of $k+1$ segments while $y$ is the concatenation of only~$k$:
the last segment of $y$ has the form $a^i_k * a^i_{k+1}$, and this meets all the requirements of Definition~\ref{def:inter} above.

 The formalisation of such a construction amounts to writing a tiny but delicately crafted computer program.\footnote{Which is even executable, if $N$ is effectively enumerable.}
 Significant effort is needed to prove fairly obvious properties of this construction, such as that the $a^i_j$ are nonempty, or that $a^{i'}_{j'} < a^i_j$ if $j'<j\le k$ and $i'$, $i<m$. These are immediate by construction since elements are drawn from the set $N$ in increasing order. But the formal proofs require fully worked out inductions.

 A similar situation arises with Larson's Lemma 3.8 \cite[p.\ts139]{Jean}. She defines three collections of sequences:
\[ \{d^j: 0<j<\omega\}, \quad \{a_j: 0<j<\omega\}, \quad \{b(i,j,k): 1\le i\le j\le k <\omega\}. \]
The ordering constraints guarantee that the set
\[ \{a_j * b(1,j,k_1) * b(2,j,k_2) * \cdots * b(j,j,k_j) : j < k_1 < k_2 < \cdots < k_j <\omega\} \]
has order type~$\omega^j$, and the union over~$j$ of these has order type~$\omega^\omega$.
These order type claims seem plausible, but the text contains no hint of how to prove them. Relevant is that the construction ensures that for $j$ and $k$ with $1\le j \le i \le k$, the sequence $b(i,j,k)$ has length $d^j_i-d^j_{i-1}$ and is therefore independent of~$k$.
For the $\omega^j$ claim, an induction on $j$ seems to be indicated; our formal proof is nearly 300 lines long, including 200 lines of auxiliary definitions and lemmas, yet a much shorter proof may exist.

It's finally time to look at the formalisation itself.
We do not present actual proofs---they are long and not especially intelligible---nor even all of the numerous definitions and technical lemmas required for the formalisation.

\subsection{Preliminary definitions and results}

We begin with something simple: the set~$W$, which is written~\isa{WW} and is the set of all strictly sorted (increasing) sequences of natural numbers. We do not work with~$W'$ except in the body of the main theorem.
\begin{isabelle}
\isacommand{definition}\ WW\ ::\ "nat\ list\ set"\isanewline
\ \ \isakeyword{where}\ "WW\ \isasymequiv \ \isacharbraceleft l.\ strict\_sorted\ l\isacharbraceright "
\end{isabelle}
Type \isa{nat~list set} is the type of sets of lists (sequences) of natural numbers.

Next comes the notion of an interaction scheme, Def.\ts\ref{def:inter} above. We need to build up to this. First, more basic Isabelle/HOL definitions:
\begin{itemize}
	\item \isa{length~l} is $\bars l$, the length of~$l$
	\item \isa{x\#l} is the list consisting of~\isa{l} prefixed with~\isa{x} as its first element
	\item \isa{u@v} is the concatenation of lists \isa{u} and~\isa{v}, like Larson's $u*v$
	\item \isa{List.set} maps a list to the corresponding finite set
\end{itemize}
The function \isa{acc\_lengths} formalises the accumulation of the lengths of lists for the variables $c$ and $d$ there. The integer argument \isa{acc} is a necessary generalisation for the sake of the recursion but will initially be zero.
\begin{isabelle}
\isacommand{fun}\ acc\_lengths\ ::\ "nat\ \isasymRightarrow \ 'a\ list\ list\ \isasymRightarrow \ nat\ list"\isanewline
\ \ \isakeyword{where}\ "acc\_lengths\ acc\ []\ =\ []"\isanewline
\ \ \ \ \ \ |\ "acc\_lengths\ acc\ (l\#ls)\ =\isanewline
\ \ \ \ \ \ \ \ \ \ \ \ \ \ (acc\ +\ length\ l)\ \#\ acc\_lengths\ (acc\ +\ length\ l)\ ls"
\end{isabelle}
Many trivial properties of \isa{acc\_lengths} must be proved. Here \isa{lists(-\isacharbraceleft []\isacharbraceright )} denotes the set of lists of nonempty lists, and the claim is that \isa{acc\_lengths} yields an element of~$W$.
\begin{isabelle}
\isacommand{lemma}\ strict\_sorted\_acc\_lengths:\isanewline
\ \ \isakeyword{assumes}\ "ls\ \isasymin \ lists\ (-\ \isacharbraceleft []\isacharbraceright )"\isanewline
\ \ \isakeyword{shows}\ "strict\_sorted\ (acc\_lengths\ acc\ ls)"
\end{isabelle}

The built-in function \isa{concat} joins a list of lists. But we also need a function to concatenate two lists of lists, interleaving corresponding elements:
\begin{isabelle}
\isacommand{fun}\ interact\ ::\ "'a\ list\ list\ \isasymRightarrow \ 'a\ list\ list\ \isasymRightarrow \ 'a\ list"\isanewline
\ \ \isakeyword{where} "interact\ []\ ys\ =\ concat\ ys"\isanewline
\ \ \ \ \ \ \ |\ "interact\ xs\ []\ =\ concat\ xs"\isanewline
\ \ \ \ \ \ \ |\ "interact\ (x\#xs)\ (y\#ys)\ =\ x\ @\ y\ @\ interact\ xs\ ys"
\end{isabelle}

We are finally ready to define interaction schemes. We split up the long list of conditions in Def.\ts\ref{def:inter} as two inductive definitions, although there is no actual induction: this form of definition works well when there are additional variables and conditions, which would otherwise have to be expressed by a big existentially quantified conjunction.
Here \isa{xs} and \isa{ys} are the $x$ and~$y$ of the definition, \isa{ka} and \isa{kb} are the lengths of the $a$-lists and $b$-lists, and \isa{zs} is the interaction scheme.
\begin{isabelle}
\isacommand{inductive}\ Form\_Body\ ::\ "[nat,nat,nat\ list,nat\ list,nat\ list]\ \isasymRightarrow \ bool"\isanewline
\ \ \isakeyword{where}\ "Form\_Body\ ka\ kb\ xs\ ys\ zs"\isanewline
\ \ \isakeyword{if}\ "length\ xs\ <\ length\ ys"\ "xs\ =\ concat\ (a\#as)"\ "ys\ =\ concat\ (b\#bs)"\isanewline
\ \ \ \ \ \ \ \ \ \ "a\#as\ \isasymin \ lists\ (-\ \isacharbraceleft []\isacharbraceright )"\ "b\#bs\ \isasymin \ lists\ (-\ \isacharbraceleft []\isacharbraceright )"\isanewline
\ \ \ \ \ \ \ \ \ \ "length\ (a\#as)\ =\ ka"\ "length\ (b\#bs)\ =\ kb"\isanewline
\ \ \ \ \ \ \ \ \ \ "c\ =\ acc\_lengths\ 0\ (a\#as)"\isanewline
\ \ \ \ \ \ \ \ \ \ "d\ =\ acc\_lengths\ 0\ (b\#bs)"\isanewline
\ \ \ \ \ \ \ \ \ \ "zs\ =\ concat\ [c,\ a,\ d,\ b]\ @\ interact\ as\ bs"\isanewline
\ \ \ \ \ \ \ \ \ \ "strict\_sorted\ zs"
\end{isabelle}

The following definition allows us to write \isa{Form~l~U} to express that~$U$ has form~$l$.
Even and odd forms are treated differently; if $l=2k$ for $k>0$, the numeric parameters of \isa{Form\_Body} are $k+1$ and~$k$. A naive treatment of the definition would force a lot of proofs to be written out twice. The two cases, while not identical, are similar enough to be treated uniformly in terms of the more general \isa{Form\_Body\ ka\ kb}.
\begin{isabelle}
\isacommand{inductive}\ Form\ ::\ "[nat,\ nat\ list\ set]\ \isasymRightarrow \ bool"\isanewline
\ \ \isakeyword{where}\ "Form\ 0\ \isacharbraceleft xs,ys\isacharbraceright "\ \isakeyword{if}\ "length\ xs\ =\ length\ ys"\ "xs\ \isasymnoteq \ ys"\isanewline
\ \ \ \ \ \ |\ "Form\ (2*k-1)\ \isacharbraceleft xs,ys\isacharbraceright "\ \isakeyword{if}\ "Form\_Body\ k\ k\ xs\ ys\ zs"\ "k>0"\isanewline
\ \ \ \ \ \ |\ "Form\ (2*k)\ \ \ \isacharbraceleft xs,ys\isacharbraceright "\ \isakeyword{if}\ "Form\_Body\ts(Suc\ k)\ k\ xs\ ys\ zs"\ "k>0"
\end{isabelle}
Finally we can define the interaction scheme itself, writing \isa{inter\_scheme~k~U} for~$i_k(U)$.
Recall that \isa{SOME} denotes Hilbert's epsilon; the \isa{zs} mentioned below 
is actually unique.
\begin{isabelle}
\isacommand{definition}\ inter\_scheme\ ::\ "nat\ \isasymRightarrow \ nat\ list\ set\ \isasymRightarrow \ nat\ list"\isanewline
\ \ \isakeyword{where}\ "inter\_scheme\ l\ U\ \isasymequiv\isanewline
\ \ \ \ \ SOME zs.\ \isasymexists k\ xs\ ys.\ l\ >\ 0\isanewline
\ \ \ \ \ \ \ \ \ \ \ \isasymand \ (l\ =\ 2*k-1\ \isasymand \ U\ =\ \isacharbraceleft xs,ys\isacharbraceright \ \isasymand \ Form\_Body\ k\ k\ xs\ ys\ zs\isanewline
\ \ \ \ \ \ \ \ \ \ \ \ \isasymor \ l\ =\ 2*k\ \isasymand \ U\ =\ \isacharbraceleft xs,ys\isacharbraceright \ \isasymand \ Form\_Body\ts(Suc\ k)\ k\ xs\ ys\ zs)"
\end{isabelle}
It turns out to be injective in the following sense, for two sets $U$ and $U'$ that have the same form $l>0$.
The proof is a painstaking reversal of the steps shown in Def.\ts\ref{def:inter} and is about 50 lines long.
\begin{isabelle}
\isacommand{proposition}\ inter\_scheme\_injective:\isanewline
\ \ \isakeyword{assumes}\ "Form\ l\ U"\ "Form\ l\ U'"\ "l\ >\ 0"\isanewline
\ \ \ \ \ \ \ \ \ \ \ "inter\_scheme\ l\ U'\ =\ inter\_scheme\ l\ U"\isanewline
\ \ \isakeyword{shows}\ "U'\ =\ U"
\end{isabelle}

A considerable effort is needed to show that the interaction scheme is defined uniquely. The following lemma takes a long set of assumptions derived from the possibility of two distinct interaction schemes and shows that the $a$-lists and $b$-lists necessarily coincide.  The proof is by induction on the length of~\isa{as}.
\begin{isabelle}
\isacommand{proposition}\ interaction\_scheme\_unique\_aux:\isanewline
\ \ \isakeyword{assumes}\ "concat\ as\ =\ concat\ as'"\ "concat\ bs\ =\ concat\ bs'"\isanewline
\ \ \ \ \ \ \isakeyword{and}\ "as\ \isasymin \ lists\ (-\ \isacharbraceleft []\isacharbraceright )"\ "bs\ \isasymin \ lists\ (-\ \isacharbraceleft []\isacharbraceright )"\isanewline
\ \ \ \ \ \ \isakeyword{and}\ "strict\_sorted\ (interact\ as\ bs)"\isanewline
\ \ \ \ \ \ \isakeyword{and}\ "length\ bs\ \isasymle \ length\ as"\ "length\ as\ \isasymle \ Suc\ (length\ bs)"\isanewline
\ \ \ \ \ \ \isakeyword{and}\ "as'\ \isasymin \ lists\ (-\ \isacharbraceleft []\isacharbraceright )"\ "bs'\ \isasymin \ lists\ (-\ \isacharbraceleft []\isacharbraceright )"\isanewline
\ \ \ \ \ \ \isakeyword{and}\ "strict\_sorted\ (interact\ as'\ bs')"\isanewline
\ \ \ \ \ \ \isakeyword{and}\ "length\ bs'\ \isasymle \ length\ as'"\ "length\ as'\ \isasymle Suc\ (length\ bs')"\isanewline
\ \ \ \ \ \ \isakeyword{and}\ "length\ as\ =\ length\ as'"\ "length\ bs\ =\ length\ bs'"\isanewline
\ \ \isakeyword{shows}\ "as\ =\ as'\ \isasymand \ bs\ =\ bs'"
\end{isabelle}

It is now fairly straightforward to show that the first four arguments of \isa{Form\_Body} determine the fifth, which is the interaction scheme. The inequality $\isa{kb}\le\isa{ka}\le\isa{kb}+1$ eliminates the need to treat the cases of even and odd forms separately.
\begin{isabelle}
\isacommand{proposition}\ Form\_Body\_unique:\isanewline
\ \ \isakeyword{assumes}\ "Form\_Body\ ka\ kb\ xs\ ys\ zs"\ "Form\_Body\ ka\ kb\ xs\ ys\ zs'"\isanewline
\ \ \ \ \ \ \isakeyword{and}\ "kb\ \isasymle \ ka"\ "ka\ \isasymle \ Suc\ kb"\isanewline
\ \ \isakeyword{shows}\ "zs'\ =\ zs"
\end{isabelle}

And so we find that the interaction scheme is uniquely defined for every valid instance of the predicate \isa{Form\_Body}. The full proof of this more-or-less obvious statement, for which Larson~\cite{Jean} gives no justification, is longer than 240 lines.
\begin{isabelle}
\isacommand{lemma}\ Form\_Body\_imp\_inter\_scheme:\isanewline
\ \ \isakeyword{assumes}\ "Form\_Body\ ka\ kb\ xs\ ys\ zs"\ "0<kb"\ "kb\isasymle ka"\ "ka\ \isasymle \ Suc\ kb"\isanewline
\ \ \isakeyword{shows}\ "zs\ =\ inter\_scheme\ ((ka+kb)\ -\ 1)\ \isacharbraceleft xs,ys\isacharbraceright "
\end{isabelle}

\subsection{Major lemmas}

The material presented thus far (plus much else not presented) serves to make sense of Larson's definitions. Now we turn to the results that make up the proof of her main result, which we sketched in Sect.\ts\ref{sec:larsonresult} above. But first, we need some Isabelle notation:
\begin{itemize}
	\item $[A]^k$, the set of all $k$-element subsets of~$A$ is \isa{[A]\isactrlbsup k\isactrlesup}
	\item $n_k$, the $k^{\text{th}}$ element of the infinite set~$N$, is \isa{enum\ N\ k}
	\item Larson's $a<b$ for ordered lists is simply \isa{a < b} in Isabelle
	\item  $(n_k) < A \subseteq N$ is the conjunction of \isa{[enum\ N\ k] < A} and \isa{A\ \isasymsubseteq \ N}.
\end{itemize}

An inductive definition expresses initial segments in terms of list concatenation, yielding
a definition of \isa{thin} sets.%
\footnote{Conceptually the same as \isa{thin\_set}, it is a property of sets of lists, not sets of sets.}
\begin{isabelle}
\isacommand{inductive}\ initial\_segment\ ::\ "'a\ list\ \isasymRightarrow \ 'a\ list\ \isasymRightarrow \ bool"\isanewline
\ \ \isakeyword{where}\ "initial\_segment\ xs\ (xs@ys)"\isanewline
\isanewline
\isacommand{definition}\ thin\ \isakeyword{where}\isanewline
\ \ "thin\ A\ \isasymequiv \ \isasymnot \ (\isasymexists x\ y.\ x\isasymin A\ \isasymand \ y\isasymin A\ \isasymand \ x\isasymnoteq y\ \isasymand \ initial\_segment\ x\ y)"
\end{isabelle}

The first of Larson's technical lemmas \cite[Lemma~3.11]{Jean} states that the set of interaction schemes $\{i_l(U) \mid U \text { has form } l \}$ is thin for $l>0$.
This set is formalised using the image operator (\isa{`}). The proof involves deriving a contradiction from the existence of $U$ and~$U'$ with distinct interaction schemes, one an initial segment of another. As with previous results, the proof involves breaking things down according to Def.\ts\ref{def:inter}, fairly straightforwardly (under 75 lines).
\begin{isabelle}
\isacommand{lemma}\ lemma\_3\_11:\isanewline
\ \ \isakeyword{assumes}\ "l\ >\ 0"\isanewline
\ \ \isakeyword{shows}\ "thin\ (inter\_scheme\ l\ `\ \isacharbraceleft U.\ Form\ l\ U\isacharbraceright )"
\end{isabelle}

The next step in Larson's development~\cite[Lemma~3.6]{Jean} is proved in 150 lines from the original 11-line text. It uses the Nash-Williams partition theorem to obtain an infinite set~$N$ and a sequence ${j_k}$ such that if $k>0$, $U$ has form $k$ and $(n_k) < i(U) \subseteq N$, then $g(U)=j_k$.
\begin{isabelle}
\isacommand{proposition}\ lemma\_3\_6:\isanewline
\ \ \isakeyword{fixes}\ g::\ "nat list set \isasymRightarrow \ nat"\isanewline
\ \ \isakeyword{assumes}\ g:\ "g\ \isasymin \ [WW]\isactrlbsup 2\isactrlesup \ \isasymrightarrow \ \isacharbraceleft 0,1\isacharbraceright "\isanewline
\ \ \isakeyword{obtains}\ N\ j\ \isakeyword{where}\ "infinite\ N"\isanewline
\ \ \ \ \isakeyword{and}\ "\isasymAnd k\ U.\ \isasymlbrakk k\ >\ 0;\ U\ \isasymin \ [WW]\isactrlbsup 2\isactrlesup ;\ Form\ k\ U;\isanewline
\ \ \ \ \ \ \ \ \ \ \ \ \ \ \ \ \ [enum\ N\ k]\ <\ inter\_scheme\ k\ U;\isanewline
\ \ \ \ \ \ \ \ \ \ \ \ \ \ \ \ \ \ List.set\ (inter\_scheme\ k\ U)\ \isasymsubseteq \ N\isasymrbrakk \ \isasymLongrightarrow \ g\ U\ =\ j\ k"
\end{isabelle}

Larson next proves that for every infinite set~$N$ and every $m$, $l<\omega$ with $l>0$, there is an $m$~element set~$M$ such that $M\subseteq W$ (necessary but omitted in the text)
and for every $\{x,y\} \subseteq M$, $\{x,y\}$ has form~$l$ and $i(\{x,y\})\subseteq N$~\cite[Lemma~3.7]{Jean}.
In the proof of the main theorem, the $N$ above is derived from the one obtained from the previous lemma. The formalisation of this one-page proof takes nearly 900 lines, including some preparatory lemmas. About 240 of those lines are devoted to establishing the basic properties of the sequences $d^1$, \ldots, $d^m$ and $a^1_1$, $a^1_2$, \ldots, $a^1_{k+1}$, $\ldots$, $a^m_1$, \ldots $a^m_{k+1}$ outlined in Sect.\ts\ref{sec:prelim-jean} above. About 130 lines were devoted to the degenerate cases $l=1$ and $l=2$, which needed to be treated separately.
\begin{isabelle}
\isacommand{proposition}\ lemma\_3\_7:\isanewline
\ \ \isakeyword{assumes}\ "infinite\ N"\ "l\ >\ 0"\isanewline
\ \ \isakeyword{obtains}\ M\ \isakeyword{where}\ "M\ \isasymin \ [WW]\isactrlbsup m\isactrlesup "\isanewline
\ \ \ \ \isakeyword{and}\ "\isasymAnd U.\ U\isasymin [M]\isactrlbsup 2\isactrlesup \ \isasymLongrightarrow \ Form\ l\ U\ \isasymand \ List.set(inter\_scheme\ l\ U)\ \isasymsubseteq \ N"
\end{isabelle}

Larson's next result states that for every infinite set~$N$, there is a set $X\subseteq W$ of order type~$\omega^\omega$ such that for any $\{x,y\} \subseteq X$, there is an $l$ such that $\{x,y\}$ has form~$l$ and if $l>0$ then $[n_l] < i(\{x,y\}) \subseteq N$. Her proof is slightly longer than a page and the full formalisation is about 1700 lines long. Of this, approximately 400 concern the construction and properties of the sequences, with a further 400 for the order type calculation and about 600 lines to formalise the last paragraph of the proof (nine lines of text). Recall that \isa{[X]\isactrlbsup 2\isactrlesup} denotes the set of all two element subsets of~$X$.
\begin{isabelle}
\isacommand{proposition}\ lemma\_3\_8:\isanewline
\ \ \isakeyword{assumes}\ "infinite\ N"\isanewline
\ \ \isakeyword{obtains}\ X\ \isakeyword{where}\ "X\ \isasymsubseteq \ WW"\ "ordertype\ X\ (lenlex\ less\_than)\ =\ \isasymomega \isasymup \isasymomega "\isanewline
\ \ \ \ \isakeyword{and}\ "\isasymAnd U.\ U\ \isasymin \ [X]\isactrlbsup 2\isactrlesup \ \isasymLongrightarrow \ \isasymexists l.\ Form\ l\ U\ \isasymand \isanewline
\ \ \ \ \ \ \ \ \ \ \ \ \ \ \ \ \ \ \ \ \ \ \ \ (l>0\ \isasymlongrightarrow \ [enum\ N\ l]\ <\ inter\_scheme\ l\ U\ \isasymand \isanewline
\ \ \ \ \ \ \ \ \ \ \ \ \ \ \ \ \ \ \ \ \ \ \ \ \ \ \ \ \ \ \ \ \ \ List.set\ (inter\_scheme\ l\ U)\ \isasymsubseteq \ N)"
\end{isabelle}

The lemmas described above constitute the main body of Larson's development. Building on them, the main theorem can be formally proved with just 360 lines of code. The formulation below---in terms of the lexicographic ordering on~$W$---trivially leads to the standard formulation in terms of the ordinal $\omega^\omega$. 

To state the main theorem, we use the Isabelle definition of a more general partition relation, $\beta\arrows (\alpha_1,\ldots, \alpha_k)^n$, which is concerned with
$n$-element subsets of~$\beta$ rather than only pairs and allows $k$ colours rather than two.
In the Isabelle version, \isa{B} is any set, \isa{r} is a well-founded relation on~\isa{B} used for order types and \isa{\isasymalpha} is a list of ordinals. 
The expression \isa{\{..<length\ \isasymalpha \}} is the set $\{0,\ldots, k-1\}$ while \isa{f\ `\ ([H]\isactrlbsup n\isactrlesup)\ \isasymsubseteq \ {i}} expresses that every element of $[H]^n$ has the colour~$i$.
\begin{isabelle}
\isacommand{definition}\ partn\_lst\ \isakeyword{where}\isanewline
\ \ "partn\_lst\ r\ B\ \isasymalpha \ n\ \isasymequiv\isanewline
\ \ \ \  \ \isasymforall f\ \isasymin \ [B]\isactrlbsup n\isactrlesup\ \ \isasymrightarrow \ \ \{..<length\ \isasymalpha \}.\isanewline
\ \ \ \ \ \ \ \ \isasymexists i\ <\ length\ \isasymalpha .\ \isasymexists H.\ H\ \isasymsubseteq \ B\ \isasymand \ ordertype\ H\ r\ =\ (\isasymalpha !i)\ \isasymand\isanewline
\ \ \ \ \ \ \ \ \ \ \ \ \ \ \ \ \ \ \ \ \ \ \ \ \ \ \ \ \ f\ `\ ([H]\isactrlbsup n\isactrlesup)\ \isasymsubseteq \ {i}"
\end{isabelle}
The proof begins by assuming a partition $f$ of the set $[W]^2$ such that there is no $m$-element set $M$ for which $f([M]^2) =\{1\}$. 
It takes about 200 lines to construct the set~$W'$ and to `replace' $W$ by~$W'$; the trick is to find a length-preserving order isomorphism between the two. The result of this work is a partition~$f'$ of $[W]^2$ that assigns the colour 0 to all pairs of sequences in $W$ of the same length, a condition necessary to make the proof go through.
The remainder of the proof, approximately 90 lines, completes the argument using the previously-proved lemmas and with no mention of~$W'$.

\begin{isabelle}
\isacommand{theorem}\ partition\_\isasymomega \isasymomega \_aux:\isanewline
\ \ \isakeyword{assumes}\ "\isasymalpha \ \isasymin \ elts\ \isasymomega "\isanewline
\ \ \isakeyword{shows}\ \ \ "partn\_lst\ (lenlex\ less\_than)\ WW\ [\isasymomega \isasymup \isasymomega ,\isasymalpha ]\ 2"\
\end{isabelle}

\subsection{On some tricky spots in the proofs}

The most frustrating aspect of formalisation is the need to spell out the proofs of obvious statements. We could not escape this phenomenon, and discuss a few examples below, some of them positive.

The function \isa{interact}, defined above, concatenates alternating elements of two lists. A key property is that if the two lists satisfy the constraints given in the definition of a form (Def.\ts\ref{def:inter}), then the result of \isa{interact} will be strictly ordered (and therefore in~$W$). The proof is a messy induction and we would like to state the theorem in the simplest possible way.

Fortunately, Isabelle/HOL provides counterexample finding tools: Nitpick~\cite{blanchette-nitpick} and Quickcheck~\cite{bulwahn-quickcheck}.
Their purpose is to identify invalid conjectures before any time is wasted in proof attempts.
Nitpick works by abstracting the conjecture to a propositional formula and attempting to find a model with the help of a satisfiability checker, while Quickcheck simply tries to evaluate the conjecture at intelligently chosen values. Both need the conjecture to be computational, in a broad sense. As much of our work here is concerned with finite sets or sequences of integers, these tools can be effective.

We were able to use Nitpick to formulate this theorem correctly, including nuances such as \isa{Suc\ n\ <\ length\ xs}, while omitting irrelevant conditions.
\begin{isabelle}
\isacommand{lemma}\ strict\_sorted\_interact\_I:\isanewline
\ \ \isakeyword{assumes}\ "length\ ys\ \isasymle \ length\ xs"\ "length\ xs\ \isasymle \ Suc\ (length\ ys)"\isanewline
\ \ \ \ "\isasymAnd x.\ x\ \isasymin \ list.set\ xs\ \isasymLongrightarrow \ strict\_sorted\ x"\isanewline
\ \ \ \ "\isasymAnd y.\ y\ \isasymin \ list.set\ ys\ \isasymLongrightarrow \ strict\_sorted\ y"\isanewline
\ \ \ \ "\isasymAnd n.\ n\ <\ length\ ys\ \isasymLongrightarrow \ xs!n\ <\ ys!n"\isanewline
\ \ \ \ "\isasymAnd n.\ Suc\ n\ <\ length\ xs\ \isasymLongrightarrow \ ys!n\ <\ xs!Suc\ n"\isanewline
\ \ \ \ "xs\ \isasymin \ lists\ (-\ \isacharbraceleft []\isacharbraceright )"\ "ys\ \isasymin \ lists\ (-\ \isacharbraceleft []\isacharbraceright )"\isanewline
\ \ \isakeyword{shows}\ "strict\_sorted\ (interact\ xs\ ys)"
\end{isabelle}

There is another challenge at the very end of Lemma 3.8 \cite[p.\ts139]{Jean}, when Larson constructs families of sequences satisfying the conditions of a form for a given pair $\{x,y\}$:
\begin{align*}
x &= \{a_j * b(1,j,k_1) * \cdots * b(j,j,k_j)\} \\
y &= \{a_r * b(1,r,p_1) * \cdots * b(r,r,p_r)\}.
\end{align*}
Here we know that $a_j<a_r$, hence $j<r$ and this turns out to guarantee the disjointness of all the segments shown. Now Larson~\cite[p.\ts140]{Jean} remarks, `for some $l<2j$, $\{x,y\}$ has form~$l$.'
Referring to the definition of form, this again is clear: the sequences for $x$ and~$y$ need to be interleaved in strict order, which may force adjacent sequences in the expression above to be concatenated. The form can be as small as 1, if $x<y$, when all the sequences get concatenated; it can be as high as $2j+1$
if no concatenations occur for~$x$, as in this example:%
\footnote{Thus it seems that Larson should have written $l<2(j+1)$. This seems to be the only error in her paper.}
\[ a_j<a_r<b(1,j,k_1)<b(1,r,p_1)<b(2,j,k_2)<\cdots <b(j,r,p_j) * \cdots * b(r,r,p_r).
\]
To ask for a justification of this obvious claim would be unreasonable. And yet to formalise the process of examining the interleavings of these sequences and arranging them so 
as to satisfy the conditions of Definition~\ref{def:inter}---so that those conditions can be \emph{proved}---turned out to require weeks of work.

Our first attempt involved the following function, which took as arguments a sequence of sequences coupled with a set~$B$, which in practice would contain the elements of the opposite sequence. The idea was to concatenate consecutive subsequences unless some element of~$B$ separated them.
\begin{isabelle}
\isakeyword{fun}\ coalesce\isanewline
\ \ \isakeyword{where}\ "coalesce [] B = []"\isanewline
\ \ \ \ \ \ \ | "coalesce [a] B = [a]"\isanewline
\ \ \ \ \ \ \ | "coalesce (a1\#a2\#as) B =\isanewline
\ \ \ \ \ \ \ \ \ \ (if \isasymexists y\isasymin B. a1 < [y] \isasymand\ [y] < a2 \isanewline
\ \ \ \ \ \ \ \ \ \ \ then a1 \# coalesce (a2\#as) B else coalesce ((a1@a2)\#as) B)"
\end{isabelle}
Nearly all the necessary properties could be proved easily, but there seemed to be no way to show that the resulting interaction scheme (obtained by applying \isa{coalesce} to both sequences of sequences) was correctly ordered. Thanks to the counterexample finder, many conjectures could be rejected without attempting a proof.

Therefore \isa{coalesce} was abandoned in favour of the following predicate, which deals with both sequences of sequences simultaneously, considering each of them cut into two parts (the arguments \isa{as1@as2} and \isa{bs1@bs2} represent arbitrary cut points for them both).
Then the two leading parts, \isa{as1} and \isa{bs1}, are concatenated provided all the ordering properties are satisfied.
\begin{isabelle}
\isacommand{inductive}\ merge\isanewline
\ \ \isakeyword{where}\ NullNull:\ "merge\ []\ []\ []\ []"\isanewline
\ \ \ \ \ \ |\ Null:\ "as1\ \isasymnoteq \ []\ \isasymLongrightarrow \ merge\ as\ []\ (concat\ as)\ []"\isanewline
\ \ \ \ \ \ |\ App:\ "\isasymlbrakk as1\ \isasymnoteq \ [];\ bs1\ \isasymnoteq \ [];\isanewline
\ \ \ \ \ \ \ \ \ \ \ \ \ \ \ concat\ as1\ <\ concat\ bs1;\ concat\ bs1\ <\ concat\ as2;\isanewline
\ \ \ \ \ \ \ \ \ \ \ \ \ \ \ merge\ as2\ bs2\ as\ bs\isasymrbrakk \isanewline
\ \ \ \
\isasymLongrightarrow \ merge\ (as1@as2)\ (bs1@bs2)\ (concat\ as1\ \#\ as)\ (concat\ bs1\ \#\ bs)"
\end{isabelle}
The conditions \isa{concat\ as1\ <\ concat\ bs1} and \isa{concat\ bs1\ <\ concat\ as2} ensure that the elements of \isa{concat\ bs1} lie between the two halves of the first sequence. Thus, just enough of the $a$-sequence is taken so that it lies before the start of the $b$-sequence, from which just enough elements are taken to allow the $a$-sequence to resume.
This formulation avoids any direct expression of iteration or computation (the root of the problems with \isa{coalesce}) in favour of writing the $a$ and $b$-sequences as each divided at an arbitrary point, the rule applying only subject to the ordering constraints shown.

With this approach, most of the required properties are shown easily enough. The most difficult is the following statement, which expresses that any two sequences---subject to certain conditions---can be successfully merged. Again, counterexample checking was crucial to find the simplest formulation of the necessary conditions. The proof is by induction on the sums of the lengths of \isa{as} and~\isa{bs}.
\begin{isabelle}
\isacommand{proposition}\ merge\_exists:\isanewline
\ \ \isakeyword{assumes}\ "strict\_sorted\ (concat\ as)"\ "strict\_sorted\ (concat\ bs)"\isanewline
\ \ \ \ \ \ \ \ \ \ "as\ \isasymin \ lists\ (-\ \isacharbraceleft []\isacharbraceright )"\ "bs\ \isasymin \ lists\ (-\ \isacharbraceleft []\isacharbraceright )"\isanewline
\ \ \ \ \ \ \ \ \ \ "hd\ as\ <\ hd\ bs"\ "as\ \isasymnoteq \ []"\ "bs\ \isasymnoteq \ []"\isanewline
\ \ \isakeyword{and}\ disj:\ "\isasymAnd a\ b.\ \isasymlbrakk a\ \isasymin \ list.set\ as;\ b\ \isasymin \ list.set\ bs\isasymrbrakk \ \isasymLongrightarrow \ a<b\ \isasymor \ b<a"\isanewline
\isakeyword{shows}\ "\isasymexists us\ vs.\ merge\ as\ bs\ us\ vs"
\end{isabelle}

\subsection{Final remarks on Larson's proof}


The formalisation of Larson's proof of $\omega^\omega\arrows(\omega^\omega, m)$ took approximately six months. This includes a month and a half spent formalising \erd–Milner~\cite{erdos-theorem-partition} and half a month proving the Nash-Williams partition theorem. Her Lemma~3.8 required two months, 11 days of which were devoted to the order type proof mentioned in Sect.\ts\ref{sec:prelim-jean} above. The remaining two months were devoted to Lemma~3.7 and the main theorem. Due to COVID-19, most of the work was undertaken at home, not the best environment for doing mathematics.

Having looked at Larson's work in excruciating detail for months, we can only be impressed by the intricacy, delicacy and fragility of her constructions and wonder how she kept so many details in mind.
She deserves her reputation of being careful and clear. Although formalisation efforts regularly identify flaws in mathematical exposition, we found no serious errors in hers.
Her narrative proof is seven pages long \cite[p.\ts133--140]{Jean} and the corresponding formalisation is some 4600 lines, not counting prerequisites such as Nash-Williams and Erd\H{o}s--Milner, which are proved elsewhere. Estimating 30 lines per page, this suggests a de Bruijn factor of roughly~23.

\section{Conclusion}\label{sec:conclusion}

Our work shows that ordinal partition theory is clearly formalisable within Isabelle/HOL augmented with a straightforward axiomatisation of ZFC. We have found no serious errors in the original mathematical material, and although we struggled in some places it is quite hard to fault  Larson's exposition~\cite{Jean} beyond noting that a few hints here and there could have saved us quite a bit of effort. The reader of a mathematical proof is expected to invest much thought. 

As usual, a concern is the disproportionate effort needed to prove some simple observations. The inductive constructions of sequences that appear in Larson's Lemmas~3.7 and~3.8 \cite{Jean} must surely be regarded as straightforward and yet we  struggled to find the right language in which to express them and derive their obvious properties. The same can be said of the order type calculation in~3.8. It's also unfortunate that the degenerate cases $l=1$ and $l=2$ in~3.7 required so much work. However, given the inherent complexity of the subject matter, it's reassuring to know that the entire development has been checked formally, with a proof text~\cite{Ordinal_Partitions-AFP} that is available for inspection or automated analysis.

This case study also demonstrates the diversity of mathematical topics that can be formalised in Isabelle/HOL: we have formalised material that is light years away from what is usually formalised.

Ordinal partition relations seem to be at the same time formalisable in Isabelle/HOL and at a point of their mathematical development where human advances seem rare and not forthcoming. None of the high power techniques of set theory and model theory such as large cardinals, forcing, pcf or elementary submodes seem to be relevant. Therefore, we hope that some advances in this subject might be obtained by automatisation. However, we remain with a humble conviction that doing enough preparatory work with Isabelle to be able to produce such results will require a considerable intellectual effort.

\paragraph*{Acknowledgements.}
 Angeliki Koutsoukou-Argyraki and Lawrence C. Paulson thank the ERC for their support through the Advanced Grant ALEXANDRIA (Project GA 742178). Mirna D{\v z}amonja's  research was supported by the GA{\v C}R project EXPRO 20-31529X and RVO: 67985840 at the Czech Academy of Sciences; she received funding from the European's Union Horizon 2020 research and innovation programme under the Maria Sko{\l}odowska-Curie grant agreement No 1010232. All three authors thank the London Mathematical Society for support through their Grant SC7-1920-11. Thanks to Stevo \stevo{} for advice
 and to the anonymous reviewers for their helpful feedback on the first
 submitted version of this paper.

\bibliographystyle{plain}  
\bibliography{Mirnamain}  

\begin{thebibliography}{10}

\bibitem{avigad-clt}
Jeremy Avigad, Johannes H{\"o}lzl, and Luke Serafin.
\newblock A formally verified proof of the central limit theorem.
\newblock {\em Journal of Automated Reasoning}, 59(4):389--423, December 2017.

\bibitem{blanchette-nitpick}
Jasmin~Christian Blanchette and Tobias Nipkow.
\newblock {Nitpick}: A counterexample generator for higher-order logic based on
  a relational model finder.
\newblock In Matt Kaufmann and Lawrence~C. Paulson, editors, {\em Interactive
  Theorem Proving}, volume 6172 of {\em Lecture Notes in Computer Science},
  pages 131--146. Springer, 2010.

\bibitem{bulwahn-quickcheck}
Lukas Bulwahn.
\newblock The new {Quickcheck} for {Isabelle}.
\newblock In Chris Hawblitzel and Dale Miller, editors, {\em Certified Programs
  and Proofs}, LNCS 7679, pages 92--108. Springer, 2012.

\bibitem{CarroyPequignot}
Rapha{\"e}l Carroy and Yann Pequignot.
\newblock Well, better and in-between.
\newblock In Monika~Seisenberger Peter~Schuster and Andreas Weiermann, editors,
  {\em Well Quasi-orders in Computation, Logic, Language and Reasoning}, pages
  1--27. Springer, 2020.

\bibitem{Changord}
Chen-Chung Chang.
\newblock A partition theorem for the complete graph on $\omega^\omega$.
\newblock {\em Journal of Combinatorial Theory (A)}, 12:396--452, 1972.

\bibitem{church40}
Alonzo Church.
\newblock A formulation of the simple theory of types.
\newblock {\em Journal of Symbolic Logic}, 5:56--68, 1940.

\bibitem{Coh66}
Paul Cohen.
\newblock {\em Set Theory and the Continuum Hypothesis}.
\newblock Benjamin, New York, 1966.

\bibitem{fasttrack}
Mirna D{\v{z}}amonja.
\newblock {\em Fast Track to Forcing}.
\newblock Cambridge University Press, 2020.

\bibitem{Erdoslist}
P.~Erd\H{o}s.
\newblock Some problems on finite and infinite graphs.
\newblock In {\em Logic and combinatorics ({A}rcata, {C}alif., 1985)},
  volume~65 of {\em Contemp. Math.}, pages 223--228. Amer. Math. Soc.,
  Providence, RI, 1987.

\bibitem{ErdosRado}
Paul Erd\H{o}s and Richard Rado.
\newblock A partition calculus in set theory.
\newblock {\em Bull. Amer. Math. Soc.}, 62(427-489), 1956.

\bibitem{erdoshajnalmaterado}
Paul Erd{\"o}s, Andr\'as Hajnal, Atilla Mate, and Richard Rado.
\newblock {\em Combinatorial {S}et {T}heory: {P}artition {R}elations for
  {C}ardinals}.
\newblock Studies in Logic and the Foundations of Mathematics 106. Elsevier
  Science Ltd, 1984.

\bibitem{erdos-theorem-partition}
Paul Erd{\H{o}}s and E.~C. Milner.
\newblock A theorem in the partition calculus.
\newblock {\em Canadian Mathematical Bulletin}, 15(4):501--505, December 1972.

\bibitem{erdos-theorem-partition-corr}
Paul Erd{\H{o}}s and E.~C. Milner.
\newblock A theorem in the partition calculus corrigendum.
\newblock {\em Canadian Mathematical Bulletin}, 17(2):305, June 1974.

\bibitem{Galvin}
Fred Galvin and Jean~A. Larson.
\newblock Pinning countable ordinals.
\newblock {\em Fundamenta Mathematicae}, 82:357--361, 1974.

\bibitem{GalvinPrikry}
Fred Galvin and Karel~L. Prikry.
\newblock Borel sets and {R}amsey's theorem.
\newblock {\em Journal of Symbolic Logic}, 38:192--198, 1973.

\bibitem{gordon-tactics-milner}
M.~J.~C. Gordon.
\newblock Tactics for mechanized reasoning: A commentary on {Milner} (1984)
  {\textquoteleft}{The} use of machines to assist in rigorous
  proof{\textquoteright}.
\newblock {\em Philosophical Transactions of the Royal Society of London A:
  Mathematical, Physical and Engineering Sciences}, 373(2039), 2015.

\bibitem{Hajnal-Larson}
Andr\'as Hajnal and Jean~A. Larson.
\newblock Partition relations.
\newblock In Matthew Foreman and Akihiro Kanamori, editors, {\em Handbook of
  Set Theory}, volume~1, pages 120--213. Springer, 2010.

\bibitem{hales-formal-Kepler}
Thomas Hales, Mark Adams, Gertrud Bauer, Tat~Dat Dang, John Harrison, Le~Truong
  Hoang, Cezary Kaliszyk, Victor Magron, Sean Mclaughlin, Tat~Thang Nguyen,
  et~al.
\newblock A formal proof of the {Kepler} conjecture.
\newblock {\em Forum of Mathematics, Pi}, 5:e2, 2017.

\bibitem{hales-jordan-curve}
Thomas~C. Hales.
\newblock The {Jordan} curve theorem, formally and informally.
\newblock {\em The American Mathematical Monthly}, 114(10):882--894, 2007.

\bibitem{harrison-exp}
John Harrison.
\newblock Floating point verification in {HOL} {L}ight: the exponential
  function.
\newblock {\em Formal Methods in System Design}, 16:271--305, 2000.

\bibitem{harrison-pnt}
John Harrison.
\newblock Formalizing an analytic proof of the prime number theorem.
\newblock {\em Journal of Automated Reasoning}, 43(3):243--261, 2009.

\bibitem{Hodkinsonnotes}
Ian Hodkinson.
\newblock Kruskal's theorem and {N}ash-{W}illiams theory, after {W}ilfrid
  {H}odges.
\newblock version 3.6 online at \url{www.doc.ic.ac.uk/~imh/papers/bar.pdf},
  February 2003.

\bibitem{hoelzl-three}
Johannes H{\"o}lzl and Armin Heller.
\newblock Three chapters of measure theory in {Isabelle/HOL}.
\newblock In Marko Eekelen, Herman Geuvers, Julien Schmaltz, and Freek Wiedijk,
  editors, {\em Interactive Theorem Proving --- Second International
  Conference}, LNCS 6898, pages 135--151. Springer, 2011.

\bibitem{jutting77}
{L.S. van Benthem} Jutting.
\newblock {\em Checking {Landau's} ``{Grundlagen}'' in the {AUTOMATH} System}.
\newblock PhD thesis, Eindhoven University of Technology, 1977.
\newblock \url{https://doi.org/10.6100/IR23183}.

\bibitem{kalvala-hol}
Sara Kalvala.
\newblock {HOL} around the world.
\newblock In M.~Archer, J.~J. Joyce, K.~N. Levitt, and P.~J. Windley, editors,
  {\em International Workshop on the {HOL} Theorem Proving System and its
  Applications}, pages 4--12. IEEE Computer Society, 1991.

\bibitem{kunen}
Kenneth Kunen.
\newblock {\em Set Theory}, volume 102 of {\em Studies in Logic and the
  Foundations of Mathematics}.
\newblock North-Holland, 1980.

\bibitem{MR1141936}
Igor K\v{r}\'{\i}\v{z} and Robin Thomas.
\newblock Analyzing {N}ash-{W}illiams' partition theorem by means of ordinal
  types.
\newblock {\em Discrete Math.}, 95(1-3):135--167, 1991.
\newblock Directions in infinite graph theory and combinatorics (Cambridge,
  1989).

\bibitem{Jean}
Jean~A. Larson.
\newblock A short proof of a partition theorem for the ordinal $\omega^\omega$.
\newblock {\em Annals of Mathematical Logic}, 6:129--145, 1973.

\bibitem{laver71}
Richard Laver.
\newblock On {Fra{\"\i}ss{\'e}}'s order type conjecture.
\newblock {\em Annals of Mathematics}, 93(1):89--111, January 1971.

\bibitem{MR1428011}
Alberto Marcone.
\newblock On the logical strength of {N}ash-{W}illiams' theorem on transfinite
  sequences.
\newblock In {\em Logic: from foundations to applications ({S}taffordshire,
  1993)}, Oxford Sci. Publ., pages 327--351. Oxford Univ. Press, New York,
  1996.

\bibitem{NashWilliamsoriginal}
C.~St. J.~A. Nash-Williams.
\newblock On well-quasi-ordering transfinite sequences.
\newblock {\em Proc. Cambridge Philos. Soc.}, 61:33--39, 1965.

\bibitem{nicely-pentium-fdiv}
Thomas~R. Nicely.
\newblock Pentium {FDIV} flaw, 2011.
\newblock FAQ page online at
  \url{https://faculty.lynchburg.edu/~nicely/pentbug/pentbug.html}.

\bibitem{isa-tutorial}
Tobias Nipkow, Lawrence~C. Paulson, and Markus Wenzel.
\newblock {\em Isabelle/HOL: A Proof Assistant for Higher-Order Logic}.
\newblock Springer, 2002.
\newblock Online at
  \url{http://isabelle.in.tum.de/dist/Isabelle/doc/tutorial.pdf}.

\bibitem{PVS96:CAV}
S.~Owre, S.~Rajan, J.M. Rushby, N.~Shankar, and M.K. Srivas.
\newblock {PVS}: Combining specification, proof checking, and model checking.
\newblock In Rajeev Alur and Thomas~A. Henzinger, editors, {\em Computer Aided
  Verification: 8th International Conference, {CAV} '96}, LNCS 1102, pages
  411--414. Springer, 1996.

\bibitem{ZFC_in_HOL-AFP}
Lawrence~C. Paulson.
\newblock Zermelo {F}raenkel set theory in higher-order logic.
\newblock {\em Archive of Formal Proofs}, October 2019.
\newblock \url{http://isa-afp.org/entries/ZFC_in_HOL.html}, Formal proof
  development.

\bibitem{Nash_Williams-AFP}
Lawrence~C. Paulson.
\newblock The {Nash-Williams} partition theorem.
\newblock {\em Archive of Formal Proofs}, May 2020.
\newblock \url{http://isa-afp.org/entries/Nash_Williams.html}, Formal proof
  development.

\bibitem{Ordinal_Partitions-AFP}
Lawrence~C. Paulson.
\newblock Ordinal partitions.
\newblock {\em Archive of Formal Proofs}, August 2020.
\newblock \url{http://isa-afp.org/entries/Ordinal_Partitions.html}, Formal
  proof development.

\bibitem{paulson-from-lcf}
Lawrence~C. Paulson, Tobias Nipkow, and Makarius Wenzel.
\newblock From {LCF} to {Isabelle/HOL}.
\newblock {\em Formal Aspects of Computing}, 31(6):675--698, 2019.

\bibitem{Ramseyth}
F.~P. Ramsey.
\newblock On a {P}roblem of {F}ormal {L}ogic.
\newblock {\em Proc. London Math. Soc. (2)}, 30(4):264--286, 1929.

\bibitem{Schipperusthesis}
Rene Schipperus.
\newblock {\em Countable partition ordinals}.
\newblock PhD thesis, University of Colorado-Boulder, 1999.

\bibitem{Schipperusjournal}
Rene Schipperus.
\newblock Countable partition ordinals.
\newblock {\em Ann. Pure Appl. Log.}, 161(10):1195--1215, 2010.

\bibitem{Speckerord}
Ernst Specker.
\newblock Teilmengen von {M}engen mit {R}elationen.
\newblock {\em Commentationes Mathematicae Helvetia}, 31:302--314, 1956.

\bibitem{Todorcevicpairs}
Stevo Todor{\v{c}}evi{\'c}.
\newblock Partitioning pairs of countable ordinals.
\newblock {\em Acta Math.}, 159(3-4):261--294, 1987.

\bibitem{stevoOCA}
Stevo Todor{\v c}evi{\'c}.
\newblock {\em Partition Problems in Topology}.
\newblock Amer. Math. Soc., Providence, Rhode Island, 1989.

\bibitem{MR2603812}
Stevo Todorcevic.
\newblock {\em Introduction to {R}amsey spaces}, volume 174 of {\em Annals of
  Mathematics Studies}.
\newblock Princeton University Press, Princeton, NJ, 2010.

\bibitem{wiedijk-de-bruijn}
Freek Wiedijk.
\newblock The {De Bruijn} factor.
\newblock Technical report, Department of Computer Science Nijmegen University,
  P.O. Box 9010, 6500 GL Nijmegen, Netherlands, 2000.
\newblock Online at \url{http://www.cs.ru.nl/~freek/factor/}.

\end{thebibliography}
\end{document}